\documentclass[11pt]{article}%
\usepackage[dvips]{graphicx}
\usepackage{amsmath}
\usepackage{amsfonts}
\usepackage{amssymb}%
\providecommand{\U}[1]{\protect\rule{.1in}{.1in}}
\providecommand{\U}[1]{\protect\rule{.1in}{.1in}}
\providecommand{\U}[1]{\protect\rule{.1in}{.1in}}
\newtheorem{theorem}{Theorem}[section]

\newtheorem{corollary}[theorem]{Corollary}

\newtheorem{lemma}[theorem]{Lemma}

\newtheorem{proposition}[theorem]{Proposition}

\newenvironment{proof}[1][Proof]{\textbf{#1.} }{\ \rule{0.5em}{0.5em}}

\begin{document}

\title{Pl\"{u}nnecke and Kneser type theorems for dimension estimates}
\date{}
\author{C\'{e}dric Lecouvey}
\maketitle

\begin{abstract}
Given a division ring $K$ containing the field $\mathrm{k}$ in its center and
two finite subsets $A$ and $B$ of $K^{\ast}$, we give some analogues of
Pl\"{u}nnecke and Kneser Theorems for the dimension of the $\mathrm{k}%
$-linear span of the Minkowski product $AB$ in terms of the dimensions of the
$\mathrm{k}$-linear spans of $A$ and $B$.\ We also explain how they imply the
corresponding more classical theorems for abelian groups. These Pl\"{u}nnecke
type estimates are then generalized to the case of associative algebras. We
also obtain an analogue in the context of division rings of a theorem by Tao
describing the sets of small doubling in a group.



\end{abstract}

\section{Introduction}

A classical problem in additive number theory is to evaluate the cardinality
of sumsets in $\mathbb{Z}$ in terms of the cardinality of their summands. Many
results and methods used to obtain such evaluations are in fact also suited
for studying the cardinality of any sumset in abelian groups. In this paper,
we will write group operations multiplicatively.\ Among numerous interesting
results in this area are the Pl\"{u}nnecke-Ruzsa and Kneser theorems.

\noindent Pl\"{u}nnecke-Ruzsa's theorem gives an upper bound for the
cardinality of $A^{n}$ knowing such a bound for $\left|  A^{2}\right|  $.

\begin{theorem}
\label{Th_plum}Let $A$ and $B$ be two finite subsets in an abelian group $G$.
Assume $\alpha$ is a positive real such that $\left\vert AB\right\vert
\leq\alpha\left\vert A\right\vert $. Then there exists a subset $X$ of $A$
such that, for any integer $n$, $\left\vert XB^{n}\right\vert \leq\alpha
^{n}\left\vert X\right\vert $.\ In particular $\left\vert A^{2}\right\vert
\leq\alpha\left\vert A\right\vert $ implies that $\left\vert A^{n}\right\vert
\leq\alpha^{n}\left\vert A\right\vert $.
\end{theorem}

\noindent Kneser's theorem gives a lower bound for the cardinality of the
product set $AB$ where $A$ and $B$ are finite nonempty subsets in an abelian
group $G$.

\begin{theorem}
Let $A$ and $B$ be finite subsets of the abelian group $G$. Write $H$ for the
stabilizer of $AB$ in $G$. Then%
\[
\left\vert AB\right\vert \geq\left\vert A\right\vert +\left\vert B\right\vert
-\left\vert H\right\vert .
\]

\end{theorem}

It is then natural to ask for analogous results in the general case of
possibly non abelian groups.\ The question of finding lower and upper bounds
for product sets in non abelian groups is considerably more difficult than in
the abelian case.\ Nevertheless, there is a growing literature on this subject
due to Diderrich \cite{Die}, Hamidoune \cite{Hami}, Kemperman \cite{Kem},
Olson \cite{Ols}, Ruzsa \cite{Ruz}, Tao \cite{Tao3} and many others. Let us
mention that Kneser's theorem does not hold for non abelian groups as noticed
in \cite{Kem}.\ Nevertheless, there exist in this case weaker versions due to
Diderrich \cite{Die} and Olson \cite{Ols}.

The Pl\"{u}nnecke-Ruzsa and Kneser theorems give informations on the structure
of groups.\ In this paper, we establish analogous results in the context of
division rings. According to the usual terminology, a field is a commutative
division ring.\ It is very easy to produce fields as extensions of simpler
ones.\ Recall that Representation Theory of groups gives also a general
procedure yielding division rings which are not necessarily fields.\ Indeed,
given a field $\mathrm{k}$ and a group $G$, the Schur lemma implies that the
endomorphism algebra $K$ of any irreducible finite-dimensional $\mathrm{k}%
[G]$-module ($\mathrm{k}[G]$ is the ring algebra of $G$ over $\mathrm{k}$) is
a division ring. When $\mathrm{k}$ is not algebraically closed, $K\supsetneqq
\mathrm{k}$ and is not commutative in general. So incidentally, we will obtain
structural information on the division ring of automorphisms of any
finite-dimensional $\mathrm{k}[G]$-module.

Let us make our notation precise. In what follows, $K$ is a division ring
containing the field $\mathrm{k}$ in its center. We address the question of
finding upper and lower bounds for the dimension of the $\mathrm{k}$-span
$\mathrm{k}\langle A_{1}\cdots A_{n}\rangle$ of product sets $A_{1}\cdots
A_{n},$ where $A_{1},\ldots,A_{n}$ are nonempty subsets of $K^{\ast
}=K\setminus\{0\}$. Note that this problem also makes sense for any algebra
$\mathcal{A}$ defined over $\mathrm{k}$.\ The estimates we obtain can thus
also be applied when $\mathcal{A}$ is contained in a division ring. In the
commutative case, this happens in particular for any algebra $\mathcal{A}%
:=\mathrm{k}[\alpha_{1},\ldots,\alpha_{n}]$, where $\alpha_{1},\ldots
,\alpha_{n}$ are elements of the field $K$, or for any sub-algebra of a field
of rational functions in several variables.

As far as we are aware, this kind of problems was considered for the first
time by Hou, Leung and Xiang \cite{Xian}, and Kainrath \cite{Kain}. Let $X$ be
a finite subset in $K$.\ We denote by $\mathrm{k}\langle X\rangle$ the
$\mathrm{k}$-subspace of $K$ generated by $X$ and write $\dim_{\mathrm{k}}(X)$
its dimension. For $X,Y$ two subsets of $K$, we set $XY=\{xy\mid x\in X,y\in
Y\}.$ Consider $A,B$ finite subsets of $K$. Then $\mathrm{k}\langle
AB\rangle=\mathrm{k}\langle\mathrm{k}\langle A\rangle\mathrm{k}\langle
B\rangle\rangle$ and $\dim_{\mathrm{k}}(AB)$ is finite. The following analogue
of Kneser's theorem for fields is proved in \cite{Xian}.

\begin{theorem}
Let $K$ be a commutative extension of $\mathrm{k}$. Assume every algebraic
element in $K$ is separable over $\mathrm{k}$.\ Let $A$ and $B$ be two
nonempty finite subsets of $K^{\ast}$. Then
\[
\dim_{\mathrm{k}}(AB)\geq\dim_{\mathrm{k}}(A)+\dim_{\mathrm{k}}(B)-\dim
_{\mathrm{k}}(H)
\]
where $H:=\{x\in K\mid x\mathrm{k}\langle AB\rangle\subseteq\mathrm{k}\langle
AB\rangle\}$.
\end{theorem}

\noindent Here $H$ is an intermediate field containing $\mathrm{k}$.
Remarkably, the authors showed that their theorem implies Kneser's theorem for
abelian groups.\ It essentially suffices to use the characterization of all
finitely generated abelian groups and the Galois correspondence. In
\cite{EL2}, we obtain an analogue of a theorem by Olson for division rings
without any separability hypothesis. In what follows, we shall refer to these
analogues as linear Kneser and linear Olson theorems. The combinatorial
methods used in the linear setting (that is, for fields or division rings) are
often very similar to their counterparts in groups. Nevertheless, there are
some restrictions and complications mostly due to the fact that

\begin{itemize}
\item $\mathrm{k}\langle A\rangle$ and $\mathrm{k}\langle A^{-1}\rangle$ do
not have the same dimension in general, whereas $A$ and $A^{-1}$ have the same cardinality,

\item a $\mathrm{k}$-subspace $V$ in $K$ may admit infinitely many
$\mathrm{k}$-subspaces $W$ such that $V\oplus W=K$, whereas a subset $A$ in a
group $G$ has a unique complement,

\item when $K$ is finite-dimensional over $\mathrm{k}$, there may exist
infinitely many intermediate division rings $H$ such that $\mathrm{k}\subseteq
H\subseteq K$, whereas a finite group $G$ has only a finite number of subgroups,

\item given $H_{1}$ and $H_{2}$ subfields of the (commutative field) $K$,
$H_{1}H_{2}$ is not a field in general, whereas the product set of two
subgroups of an abelian group is always a group.
\end{itemize}

\noindent So, to avoid gaps or ambiguities, we have completely written down
the proofs of our linear statements.\ These proofs sometimes differ from their
analogues in groups. For example, the possible existence of an infinite number
of intermediate fields seems to impose a separability hypothesis in the
previous linear Kneser theorem. It is nevertheless conjectured in \cite{Hou}
that this hypothesis can be relaxed as in the linear Olson theorem. Also, in
order to adapt the arguments used to establish the estimates in groups, we
often need, in our division ring context, to carefully chose the decomposition
in direct summands of the spaces we consider in our proofs.

\bigskip

The paper is organized as follows. In Section 2, we make our notation precise
and give equivalent forms of the linear Kneser theorem.\ Section 3 is devoted
to some linear analogues of results by Ruzsa. In particular, we derive a
Pl\"{u}nnecke-Ruzsa type theorem for fields.\ The arguments we use here are
adaptations to the context of division rings of some very elegant and
elementary proofs recently obtained by Petridis in \cite{Petr}. We also
establish that our Pl\"{u}nnecke-Ruzsa theorem for fields implies the
corresponding theorem for abelian groups. In Section 4, we generalize the
Pl\"{u}nnecke-Ruzsa type theorems of Section 2 in the context of associative
unital algebras. In Section 5, we establish different Kneser type estimates
for division rings. More precisely, we first study the case where $A$ is
assumed commutative (that is, the elements of $A$ are pairwise commutative)
and obtain a linear analogue of a theorem by Diderrich \cite{Die}. Finally, we
adapt Hamidoune connectivity to the context of division rings and obtain a
linear version of a theorem by Tao describing the sets of small doubling in a group.

\section{The division ring setting}

\subsection{Vector span in a division ring}

Let $K$ be a division ring and $\mathrm{k}$ a subfield (thus commutative) of
$K$ contained in its center. We denote by $K^{\ast}=K\setminus\{0\}$ the group
of invertible elements in $K$.

For any subset $A$ of $K^{\ast}$, let $\mathrm{k}\langle A\rangle$ be the
$\mathrm{k}$-subspace of $K$ generated by $A$. We write $\left\vert
A\right\vert $ for the cardinality of $A$, and $\dim_{\mathrm{k}}(A)$ for the
dimension of $\mathrm{k}\langle A\rangle$ over $\mathrm{k}$.\ When $\left\vert
A\right\vert $ is finite, $\dim_{\mathrm{k}}(A)$ is also finite and we have
$\dim_{\mathrm{k}}(A)\leq\left\vert A\right\vert $. We denote by
$\mathbb{D}(A)\subseteq K$ the sub division ring generated by $A$ in $K$.

Given subsets $A$ and $B$ of $K$, we thus have $\mathrm{k}\langle A\cup
B\rangle=\mathrm{k}\langle A\rangle+\mathrm{k}\langle B\rangle$, the sum of
the two spaces $\mathrm{k}\langle A\rangle$ and $\mathrm{k}\langle B\rangle$.
We have also $\mathrm{k}\langle A\cap B\rangle\subseteq\mathrm{k}\langle
A\rangle\cap\mathrm{k}\langle B\rangle$ and $\mathrm{k}\langle AB\rangle
=\mathrm{k}\langle\mathrm{k}\langle A\rangle\mathrm{k}\langle B\rangle\rangle
$.\ We write as usual
\[
AB:=\{ab\mid a\in A,b\in B\}
\]
for the Minkowski product of the sets $A$ and $B$.\ Given a family of nonempty
subsets $A_{1},\ldots,A_{n}$ of $K^{\ast}$, we define $A_{1}\cdots A_{n}$
similarly.\ We also set $A^{-1}:=\{a^{-1}\mid a\in A\}$. Observe that any
finite-dimensional $\mathrm{k}$-subspace $V$ of $K$ can be realized as
$V=\mathrm{k}\langle A\rangle$, where $A$ is any finite subset of nonzero
vectors spanning $V$. Also, when $V_{1}$ and $V_{2}$ are two $\mathrm{k}%
$-vector spaces in $K$, $V_{1}V_{2}\subseteq\mathrm{k}\langle V_{1}%
V_{2}\rangle$ but $V_{1}V_{2}$ is not a vector space in general.

In what follows we aim to give some estimates of $\dim_{\mathrm{k}}(AB)$ or
more generally of $\dim_{\mathrm{k}}(A_{1}\cdots A_{r})$ where $A_{1}%
,\ldots,A_{r}$ are finite subsets of $K^{\ast}$. The following is
straightforward%
\[
\max(\dim_{\mathrm{k}}(A),\dim_{\mathrm{k}}(B))\leq\dim_{\mathrm{k}}%
(AB)\leq\dim_{\mathrm{k}}(A)\dim_{\mathrm{k}}(B).
\]
The methods we will use to estimate $\dim_{\mathrm{k}}(AB)$ are quite
analogous to the tools used to estimate the cardinality $\left\vert
AB\right\vert $ of the product set $AB$ where $A$ and $B$ are subsets of a
given group. Many results on estimates of product sets have a linear analogue
for the dimension of the space generated by such products. Nevertheless, there
are crucial differences due notably to the fact that $\mathrm{k}\langle
A\rangle$ and $\mathrm{k}\langle A^{-1}\rangle$ have not the same dimension in
general, whereas $A$ and $A^{-1}$ has the same cardinality. This can be easily
verified by taking $\mathrm{k}=\mathbb{C}$, $K=\mathbb{C}(t)$ the function
field in the indeterminate $t$ and $A_{n}:=\{(t-a)\mid a=1,\ldots,n\}$.
Indeed, $\dim_{\mathbb{C}}(A_{n})=2$, but $\dim_{\mathbb{C}}(A_{n}^{-1}%
)=n$.\ Also observe that a finite abelian group has only a finite number of
subgroups, whereas a finite commutative extension of a commutative field have
an infinite number of intermediate extensions when it is not separable.

The two following elementary lemmas will be useful.

\begin{lemma}
\label{lem_stab}Let $A$ be a finite subset of $K^{\ast}$ such that
$\mathrm{k}\langle A^{2}\rangle=\mathrm{k}\langle A\rangle$.\ Then
$\mathrm{k}\langle A\rangle$ is a division ring.
\end{lemma}

\begin{proof}
Observe that $\mathrm{k}\langle A^{2}\rangle=\mathrm{k}\langle A\rangle$ means
that $\mathrm{k}\langle A\rangle$ is closed under multiplication. Then, for
any nonzero $a\in\mathrm{k}\langle A\rangle$, the map $\varphi_{a}%
:\mathrm{k}\langle A\rangle\rightarrow\mathrm{k}\langle A\rangle$ which sends
$\alpha\in\mathrm{k}\langle A\rangle$ on $\varphi_{a}(\alpha)=a\alpha$ is a
$\mathrm{k}$-linear automorphism of the space $\mathrm{k}\langle A\rangle
$.\ In particular, $\varphi_{a}$ is surjective. Since $a\in\mathrm{k}\langle
A\rangle$, $1$ belongs to $\mathrm{k}\langle A\rangle$. Now since
$1\in\mathrm{k}\langle A\rangle$, $\alpha=a^{-1}\in A$ and $\mathrm{k}\langle
A\rangle$ is a field.
\end{proof}

\begin{lemma}
\label{lemma_stab}Consider a finite-dimensional $\mathrm{k}$-subspace $V$ of
$K$ and a sub division ring $H$ of $K$ containing $\mathrm{k}$ such that
$HV=V$. Then there exists a finite subset $S$ of $K^{\ast}$ such that
\begin{equation}
V=\bigoplus_{s\in S}Hs\text{.} \label{decom}%
\end{equation}
In particular, $V$ is a left $H$-module of dimension $\left\vert S\right\vert
\dim_{\mathrm{k}}(H)$, and (\ref{decom}) gives its decomposition into
irreducible components.
\end{lemma}

\begin{proof}
For any nonzero vector $v$ in $V$, $Hv$ is a $\mathrm{k}$-subspace of $V$. In
particular $\dim_{\mathrm{k}}(Hv)=\dim_{\mathrm{k}}(H)$ is finite. Moreover if
$s$ is a nonzero vector in $V\setminus(Hs_{1}\oplus\cdots\oplus Hs_{t})$, we
have $Hs\cap(Hs_{1}\oplus\cdots\oplus Hs_{t})=\{0\}$.\ The lemma easily follows.
\end{proof}

\bigskip

\noindent\textbf{Remarks:}

\begin{enumerate}
\item From the previous lemma applied to $V=K$, we deduce that $\dim
_{\mathrm{k}}(H)$ divides $\dim_{\mathrm{k}}(K)$ when $\dim_{\mathrm{k}}(K)$
is finite.

\item When $VH=V$, we similarly obtain a decomposition into right $H$-modules
$V=\bigoplus_{s\in S}sH$.
\end{enumerate}

\bigskip

For any subset $X$ in $K^{\ast}$, we set%
\[
H_{\mathrm{k},l}(X):=\{h\in K\mid h\,\mathrm{k}\langle X\rangle\subseteq
\mathrm{k}\langle X\rangle\}\text{ and }H_{\mathrm{k},r}(X):=\{h\in
K\mid\mathrm{k}\langle X\rangle\,h\subseteq\mathrm{k}\langle X\rangle\}
\]
for the left and right stabilizers of $\mathrm{k}\langle X\rangle$ in $K$.
Clearly $H_{\mathrm{k},l}(X)$ and $H_{\mathrm{k},r}(X)$ are division rings
containing $\mathrm{k}$.\ In particular, when $K$ is a field, $H_{\mathrm{k}%
,l}(X)=H_{\mathrm{k},r}(X)$ is a commutative extension of $\mathrm{k}$ that we
simply write as $H_{\mathrm{k}}(X)$.\ If $H_{\mathrm{k},l}(X)$ (resp.
$H_{\mathrm{k},r}(X)$) is not equal to $\mathrm{k}$, we says that
$\mathrm{k}\langle X\rangle$ is \emph{left periodic} (resp. \emph{right
periodic}). When $\mathrm{k}\langle X\rangle$ is finite-dimensional,\ there
exists by Lemma \ref{lemma_stab} a finite subset $S$ in $\mathrm{k}\langle
X\rangle$ such that $\mathrm{k}\langle X\rangle=\oplus_{s\in S}H_{\mathrm{k}%
,l}(X)s$ (resp. $\mathrm{k}\langle X\rangle=\oplus_{s\in S}sH_{\mathrm{k}%
,r}(X)$). Observe also for $X$ and $Y$ two finite subsets of $K$, if $\langle
X\rangle$ is left periodic (resp. $\langle Y\rangle$ is right periodic), then
$\langle XY\rangle$ is left periodic (resp. right periodic). Indeed, for
$\langle X\rangle$ left periodic, we have $H_{\mathrm{k},l}(X)\neq\mathrm{k}$
and $H_{\mathrm{k},l}(X)\langle X\rangle\subseteq\langle X\rangle$.\ By
linearity of the multiplication in $K$, this gives $H_{\mathrm{k},l}(X)\langle
XY\rangle\subseteq\langle XY\rangle$ thus $H_{\mathrm{k},l}(X)\subseteq
H_{\mathrm{k},l}(XY)$ and $H_{\mathrm{k},l}(XY)\neq\mathrm{k}$.\ The case
$\langle Y\rangle$ right periodic is similar.

\subsection{Product sets in abelian group and linear setting}

\label{subsec-linkabelina}As observed in \cite{Xian}, estimates for
cardinality of product sets in abelian groups can be obtained from their
linear counterparts in fields. To do this consider an abelian group $G$ and
assume%
\[
G\simeq\mathbb{Z}^{l}\times\mathbb{Z}/t_{1}\mathbb{Z\times}\cdots
\times\mathbb{Z}/t_{r}\mathbb{Z}\text{.}%
\]
The idea is to construct a field extension $K$ of $\mathrm{k}=\mathbb{C}$ such
that the multiplicative group of $K$ contains a subgroup isomorphic to $G$
whose elements are linearly independent over $\mathbb{C}$. Consider first the
field $E=\mathbb{C}(x_{1},\ldots,x_{r})$ of rational functions over the
indeterminates $x_{1},\ldots,x_{r}$ ($x_{1},\ldots,x_{r}$ are thus assumed
algebraically independent).\ Let $F$ be the decomposition field of the
polynomial
\[
P(Z)=(Z-x_{1}^{t_{1}})\cdots(Z-x_{r}^{t_{r}})\in E[Z]
\]
over $E$.\ There exist elements $\sqrt[t_{1}]{x_{1}},\ldots,\sqrt[t_{r}]%
{x_{r}}$ in $F$ such that $F=E[\sqrt[t_{1}]{x_{1}},\ldots,\sqrt[t_{r}]{x_{r}%
}]$. Now set $K=F(T_{1},\ldots,T_{l})$ the field of rational functions over
$E$ in the algebraically independent indeterminates $T_{1},\ldots,T_{l}$.

There exists in $G$ non torsion elements $u_{1},\ldots,u_{l}$ and elements
$\gamma_{1},\ldots,\gamma_{r}$ respectively of order $t_{1},\ldots,t_{r}$ such
that each element $g\in G$ can be written uniquely on the form $g=u_{1}%
^{a_{1}}\cdots u_{l}^{a_{l}}\gamma_{1}^{b_{1}}\cdots\gamma_{r}^{b_{r}}$ where
$a_{i}\in\mathbb{Z}$ for any $i=1,\ldots,l$ and $b_{j}\in\{0,\ldots,t_{j}-1\}$
for any $j=1,\ldots,r$. One associates to $g\in G$ the element $\eta(g)\in
K\setminus\{0\}$ such that%
\[
\eta(g)=T_{1}^{a_{1}}\cdots T_{l}^{a_{l}}(\sqrt[t_{1}]{x_{1}})^{b_{1}}%
\cdots(\sqrt[t_{r}]{x_{r}})^{b_{r}}.
\]
Clearly $\eta$ is a group embedding from $G$ to the multiplicative group of
$K$ and the elements of $\eta(G)$ are linearly independent over $\mathbb{C}$
in $K$. Consider a finite subset $X\subseteq G$. Let $\Phi(X)=\mathbb{C}%
\langle\eta(X)\rangle$. We also define
\[
S(X)=\{g\in G\mid gX=X\}\text{ and }H_{\mathbb{C}}(X)=\{x\in K\mid
x\Phi(X)\subseteq\Phi(X)\}\text{.}%
\]
Observe that $S(X)$ is a finite subgroup of $X$ and $H_{\mathbb{C}%
}(X)\subseteq K$ is a field extension of $\mathbb{C}$.

We refer to \cite{Xian} for the proof of the following proposition (which uses
Galois correspondence for Assertion 3). This proposition provides a
correspondence between cardinality and dimension which allows one to recover
classical results from their dimensional counterparts. This applies for
instance to Theorem 2.7, or Theorem 3.3.'

\begin{proposition}
\label{prop_link}Given $X,Y$ two finite subsets in an abelian $G$, we have

\begin{enumerate}
\item $\left\vert X\right\vert =\dim_{\mathbb{C}}(\Phi(X))$,

\item $\left\vert XY\right\vert =\dim_{\mathbb{C}}(\Phi(X)\Phi(Y)),$

\item $\Phi(S(X))=H_{\mathbb{C}}(X)$.
\end{enumerate}
\end{proposition}

\subsection{Product sets in nonabelian groups and linear setting}

\label{subsec-linknonabelina}It would be desirable to have an analogue of
Proposition \ref{prop_link} in the non abelian setting.\ Unfortunately, the
construction of the previous paragraph can not be generalized. Indeed, there
are finite subgroups which cannot be embedded in the group of invertible of a
division ring (see \cite{Am}).

In this paragraph, we will focus on a weaker connection between cardinality of
product sets in a (possible non abelian) group $G$ and dimension of subspaces
in the group algebra $\mathbb{C}[G]$.

Consider $G$ a (possible non abelian) group. Recall that the group algebra
$\mathbb{C}[G]$ is the associative algebra generated over $\mathbb{C}$ by the
elements $e_{g},g\in G$ and the relations $e_{g}e_{g^{\prime}}=e_{gg^{\prime}%
}$ for any $g,g^{\prime}\in G$. To any finite subset $X$ of $G$ we associate
the $\mathbb{C}$-subspace $\mathrm{k}\langle A_{X}\rangle$ of $\mathbb{C}[G]$
where $A_{X}=\{e_{g}\mid g\in X\}\subset\mathbb{C}[G]$.

Given two finite subspaces $X$ and $Y$ of $\mathbb{C}[G]$, we have clearly%
\[
\left\vert X\right\vert =\dim_{\mathbb{C}}(A_{X})\text{ and }\left\vert
XY\right\vert =\dim_{\mathbb{C}}(A_{X}A_{Y}).
\]
In Section \ref{Sec_alge}, we will see that it is possible to recover some
estimates for the cardinality of product sets in groups from analogous
statements in ring algebras.

\subsection{The linear Kneser theorem}

We now recall the linear Kneser theorem stated in \cite{Xian} for fields.

\begin{theorem}
\label{Th_HL}Let $K$ be a commutative extension of $\mathrm{k}$. Assume every
algebraic element in $K$ is separable over $\mathrm{k}$.\ Let $A$ and $B$ be
two nonempty, finite subsets of $K^{\ast}$. Then
\[
\dim_{\mathrm{k}}(AB)\geq\dim_{\mathrm{k}}(A)+\dim_{\mathrm{k}}(B)-\dim
_{\mathrm{k}}(H_{\mathrm{k}}(AB)).
\]

\end{theorem}

In \S \ \ref{subsec-linTao}, we will give a noncommutative version of the
following corollary due to Tao \cite{Tao} where the separability hypothesis
can be relaxed.

\begin{corollary}
Let $K$ be a commutative extension of $\mathrm{k}$. Assume every algebraic
element in $K$ is separable over $\mathrm{k}$.\ Let $A$ be a nonempty, finite
subset of $K^{\ast}$ such that $\dim_{\mathrm{k}}(A^{2})\leq(2-\varepsilon
)\dim_{\mathrm{k}}(A)$ for a real $\varepsilon$ with $0<\varepsilon\leq1$.
Then there exists a field $H$ that is finite-dimensional over $\mathrm{k}$ and
a finite non empty, subset $X$ of $K^{\ast}$ with $\left\vert X\right\vert
\leq\frac{2}{\varepsilon}-1$ such that $\mathrm{k}\langle A^{2}\rangle
\subseteq\bigoplus_{x\in X}xH$.
\end{corollary}

\begin{proof}
By the previous theorem, we must have $\dim_{\mathrm{k}}(H_{\mathrm{k}}%
(A^{2}))\geq2\dim_{\mathrm{k}}(A)-\dim_{\mathrm{k}}(A^{2})\geq\varepsilon
\dim_{\mathrm{k}}(A)$. By Lemma \ref{lemma_stab}, there exists a finite subset
$X$ of $K^{\ast}$ such that $\mathrm{k}\langle A^{2}\rangle=\bigoplus_{x\in
X}xH$. We thus have
\[
\left\vert X\right\vert =\frac{\dim_{\mathrm{k}}(A^{2})}{\dim_{\mathrm{k}}%
(H)}\leq\frac{(2-\varepsilon)\dim_{\mathrm{k}}(A)}{\varepsilon\dim
_{\mathrm{k}}(A)}\leq\frac{2}{\varepsilon}-1
\]
as desired.
\end{proof}

\bigskip

\noindent\textbf{Remarks:}

\begin{enumerate}
\item Theorem \ref{Th_HL} can be regarded as a linear version of Kneser's
theorem.\ Recall that this theorem establishes that for any nonempty, finite
subsets $A$ and $B$ in an abelian group $G$ (written multiplicatively), we
have $\left\vert AB\right\vert \geq\left\vert A\right\vert +\left\vert
B\right\vert -\left\vert H\right\vert $, where $H$ is the stabilizer of $AB$
in $G$.\ 

\item As proved in \cite{Xian}, the linear Kneser theorem implies Kneser's
theorem. This follows from Proposition \ref{prop_link}.

\item The separability hypothesis is crucial in the proof of the theorem as
the finite extensions of $\mathrm{k}$ should have a finite number of
intermediate extensions. Nevertheless, it is conjectured in \cite{Hou} that
the separability hypothesis can be relaxed. Also observe that the separability
hypothesis is always satisfied in characteristic zero.
\end{enumerate}

Like the original Kneser theorem, Theorem \ref{Th_HL} can be generalized for
Minkowski products of any finite number of finite subsets of $K^{\ast}$. The
following theorem is not explicitly stated in \cite{Xian}. We give its proof
below for completion. We first need the following easy lemma.

\begin{lemma}
\label{lemma_util}Let $K$ be a commutative extension of $\mathrm{k}$. Consider
an integer $n\geq2$ and a collection of finite, nonempty subsets $A_{1}%
,\ldots,A_{n}$ of $K^{\ast}$ such that
\begin{equation}
\dim_{\mathrm{k}}(\prod_{i=1}^{j}A_{i})\geq\dim_{\mathrm{k}}(\prod_{i=1}%
^{j-1}A_{i})+\dim_{\mathrm{k}}(A_{j})-1 \label{Hr}%
\end{equation}
holds for $j=2,\ldots,n$.\ Then
\[
\dim_{\mathrm{k}}(\prod_{i=1}^{n}A_{i})\geq\sum_{i=1}^{n}\dim_{\mathrm{k}%
}(A_{i})-n+1\text{.}%
\]

\end{lemma}

\begin{proof}
We proceed by induction on $j$.\ For $j=2$, we have $\dim_{\mathrm{k}}%
(A_{1}A_{2})\geq\dim_{\mathrm{k}}(A_{1})+\dim_{\mathrm{k}}(A_{2})-1$ by
(\ref{Hr}). Assume we have%
\begin{equation}
\dim_{\mathrm{k}}(\prod_{i=1}^{j}A_{i})\geq\sum_{i=1}^{j}\dim_{\mathrm{k}%
}(A_{i})-j+1\text{.} \label{indu}%
\end{equation}
Writing (\ref{Hr}) with $j+1$ gives%
\[
\dim_{\mathrm{k}}(\prod_{i=1}^{j+1}A_{i})\geq\dim_{\mathrm{k}}(\prod_{i=1}%
^{j}A_{i})+\dim_{\mathrm{k}}(A_{j+1})-1.
\]
Combining with (\ref{indu}), one obtains%
\[
\dim_{\mathrm{k}}(\prod_{i=1}^{j+1}A_{i})\geq\sum_{i=1}^{j}\dim_{\mathrm{k}%
}(A_{i})-j+\dim_{\mathrm{k}}(A_{j+1})
\]
as desired.
\end{proof}

\begin{theorem}
\label{Th_Lin_Kne_n}Let $K$ be a commutative extension of $\mathrm{k}$.
Consider a collection of finite, nonempty subsets $A_{1},\ldots,A_{n}$ of
$K^{\ast}$ Set $H:=H_{\mathrm{k}}(A_{1}\cdots A_{n})$. The following
statements are equivalent:

\begin{enumerate}
\item $\dim_{\mathrm{k}}(A_{1}\cdots A_{n})\geq\sum_{i=1}^{n}\dim_{\mathrm{k}%
}(A_{i}H)-(n-1)\dim_{\mathrm{k}}(H)$,

\item $\dim_{\mathrm{k}}(A_{1}\cdots A_{n})\geq\sum_{i=1}^{n}\dim_{\mathrm{k}%
}(A_{i})-(n-1)\dim_{\mathrm{k}}(H)$,

\item either $\dim_{\mathrm{k}}(A_{1}\cdots A_{n})\geq\sum_{i=1}^{n}%
\dim_{\mathrm{k}}(A_{i})-(n-1)$ or $\mathrm{k}\langle A_{1}\cdots A_{n}%
\rangle$ is periodic,

\item any one of the above three statements in the case $n=2$.
\end{enumerate}
\end{theorem}

\begin{proof}
(I): We obtain $1\Rightarrow2$ by using $\dim_{\mathrm{k}}(A_{i}H)\geq
\dim_{\mathrm{k}}(A_{i})$ for any $i=1,\ldots,n$.\ The implication
$2\Longrightarrow3$ is immediate.\ To prove $3\Longrightarrow1$, we first
observe the implication is true when $H=\mathrm{k}$. Note that $H(A_{1}\cdots
A_{n})$ is, by definition, not periodic. When $\mathrm{k}\langle A_{1}\cdots
A_{n}\rangle$ is periodic using the base field $\mathrm{k}$, it is not
periodic using the base field $H$. We thus can apply assertion $3$ by
considering $K$ as a commutative extension of $H$. This gives $\dim_{H}%
(A_{1}\cdots A_{n})\geq\sum_{i=1}^{n}\dim_{H}(A_{i})-(n-1)$.\ Multiplying this
inequality by $\dim_{\mathrm{k}}(H)$ yields%
\[
\dim_{H}(A_{1}\cdots A_{n})\dim_{\mathrm{k}}(H)\geq\sum_{i=1}^{n}\dim
_{H}(A_{i})\dim_{\mathrm{k}}(H)-(n-1)\dim_{\mathrm{k}}(H)
\]
which is equivalent to $1$ since $\dim_{H}(A_{1}\cdots A_{n})\dim_{\mathrm{k}%
}(H)=\dim_{\mathrm{k}}(A_{1}\cdots A_{n})$ and $\dim_{H}(A_{i})\dim
_{\mathrm{k}}(H)=\dim_{\mathrm{k}}(A_{i}H)$ for any $i=1,\ldots,n$.

(II): It remains to prove that assertion $3$ is equivalent to the following:

\begin{enumerate}
\item[3'] Given two finite, nonempty subsets $A$ and $B$ of $K^{\ast}$, either
$\dim_{\mathrm{k}}(AB)\geq\dim_{\mathrm{k}}(A)+\dim_{\mathrm{k}}(B)-1$ or
$\mathrm{k}\langle AB\rangle$ is periodic$.$
\end{enumerate}

Clearly $3\Longrightarrow3^{\prime}$.\ Now assume $3^{\prime}$ holds and
$\dim_{\mathrm{k}}(A_{1}\cdots A_{n})<\sum_{i=1}^{n}\dim_{\mathrm{k}}%
(A_{i})-(n-1)$.\ Then, by Lemma \ref{lemma_util}, there exists $j\in
\{2,\ldots,n\}$ such that
\[
\dim_{\mathrm{k}}(\prod_{i=1}^{j}A_{i})<\dim_{\mathrm{k}}(\prod_{i=1}%
^{j-1}A_{i})+\dim_{\mathrm{k}}(A_{j})-1.
\]
By applying $3^{\prime}$ with $A=\prod_{i=1}^{j-1}A_{i}$ and $B=A_{j}$, we
obtain that $\mathrm{k}\langle AB\rangle=\mathrm{k}\langle\prod_{i=1}^{j}%
A_{i}\rangle$ is periodic. Therefore $\mathrm{k}\langle\prod_{i=1}^{n}%
A_{i}\rangle$ is also periodic. This shows that $3^{\prime}\Longrightarrow3$.
\end{proof}

\bigskip

\noindent\textbf{Remark:} The four assertions of the theorem are equivalent
without the separability hypothesis of Theorem \ref{Th_HL}.

\section{Pl\"{u}nnecke-type estimates in division rings}

\label{Sec_Plum_K}

Pl\"{u}nnecke-type estimates permit one to obtain upper bounds on the
cardinality of sumsets in abelian groups as stated in Theorem \ref{Th_plum}.
Pl\"{u}nnecke's result was first stated for $G=\mathbb{Z}$, but his proof,
based on a graph-theoretic method, can be extended to arbitrary abelian
groups. Very recently, Petridis gave a surprisingly elegant and short proof of
Theorem \ref{Th_plum}.\ This proof can be adapted to the context of division
rings as explained in the following paragraphs. In Section \ref{Sec_alge}, we
will consider the case of associative unital algebras.

\subsection{Minimal growth under multiplication}

Let $K$ be a division ring containing the field $\mathrm{k}$ in its center.
Consider two finite subsets $A$ and $B$ of $K^{\ast}$. For any $\mathrm{k}%
$-subspace $V\neq\{0\}$ of $\mathrm{k}\langle A\rangle$ (hence $V$ is
finite-dimensional), we set
\begin{equation}
r(V):=\frac{\dim_{\mathrm{k}}(VB)}{\dim_{\mathrm{k}}(V)}\label{r(Z)}%
\end{equation}
for the growth of $V$ under multiplication by $B$. Write $\rho:=\min
_{V\subseteq\mathrm{k}\langle A\rangle,V\neq\{0\}}r(V)$. Since the image of
the map $r$ is contained in a discrete set of positive numbers, there exists a
nonempty set $X\subseteq\mathrm{k}\langle A\rangle$ such that $r(\mathrm{k}%
\langle X\rangle)=\rho$. We thus have, $\dim_{\mathrm{k}}(XB)=\rho
\dim_{\mathrm{k}}(X)$ and $\dim_{\mathrm{k}}(XB)/\dim_{\mathrm{k}}(X)\leq
\dim_{\mathrm{k}}(ZB)/\dim_{\mathrm{k}}(Z)$ for any $Z\subseteq\mathrm{k}%
\langle A\rangle$.

\begin{proposition}
\label{prop-CXB}Under the previous hypotheses, we have, for any finite set $C$
in $K^{\ast}$,%
\[
\dim_{\mathrm{k}}(CXB)\leq\rho\dim_{\mathrm{k}}(CX)=\frac{\dim_{\mathrm{k}%
}(CX)\dim_{\mathrm{k}}(XB)}{\dim_{\mathrm{k}}(X)}.
\]

\end{proposition}

\begin{proof}
Write $C=\{c_{1},\ldots,c_{r}\}$.\ We construct by induction subsets
$X_{1},\ldots,X_{r}$ of $X$ such that%
\[
\sum_{i=1}^{j}c_{i}\mathrm{k}\langle X\rangle=\bigoplus_{i=1}^{j}%
c_{i}\mathrm{k}\langle X_{i}\rangle
\]
for any $j=1,\ldots r$. Set $X_{1}=X$. Now assume we have constructed subsets
$X_{1},\ldots,X_{j-1}$ of $X$ such that%
\[
\sum_{i=1}^{j-1}c_{i}\mathrm{k}\langle X\rangle=\bigoplus_{i=1}^{j-1}%
c_{i}\mathrm{k}\langle X_{i}\rangle.
\]
We have
\[
\sum_{i=1}^{j}c_{i}\mathrm{k}\langle X\rangle=\bigoplus_{i=1}^{j-1}%
c_{i}\mathrm{k}\langle X_{i}\rangle+c_{j}\mathrm{k}\langle X\rangle\text{.}%
\]
Let $B_{j-1}$ be any basis of $\bigoplus_{i=1}^{j-1}c_{i}\mathrm{k}\langle
X_{i}\rangle$. Then $B_{j-1}\cup c_{j}X$ generates $\sum_{i=1}^{j}%
c_{i}\mathrm{k}\langle X\rangle$. There thus exists a subset $X_{j}\subseteq
X$ such that

$B_{j-1}\cup c_{j}X_{j}$ is a basis of $\sum_{i=1}^{j}c_{i}\mathrm{k}\langle
X\rangle$. We then have%
\[
\sum_{i=1}^{j}c_{i}\mathrm{k}\langle X\rangle=\bigoplus_{i=1}^{j}%
c_{i}\mathrm{k}\langle X_{i}\rangle
\]
as desired.

By construction of the sets $X_{i}$, we have
\begin{equation}
\dim_{\mathrm{k}}(\sum_{i=1}^{j}c_{i}\mathrm{k}\langle X\rangle)=\sum
_{i=1}^{j}\dim_{\mathrm{k}}(c_{i}X_{i})=\sum_{i=1}^{j}\dim_{\mathrm{k}}(X_{i})
\label{deco}%
\end{equation}
for any $j=1,\ldots,r$.

As in the proof of Petridis, we now proceed by induction on $r$. When $r=1$,
$\dim_{\mathrm{k}}(c_{1}XB)=\dim_{\mathrm{k}}(XB)=\rho\dim_{\mathrm{k}%
}(X)=\rho\dim_{\mathrm{k}}(c_{1}X)$. Assume $r>1$. Set $V_{r}:=c_{r}^{-1}%
\sum_{a=1}^{r-1}c_{a}\mathrm{k}\langle X\rangle\cap\mathrm{k}\langle X\rangle$
and $\mathrm{k}\langle X\rangle=\mathrm{k}\langle X_{r}\rangle\oplus V_{r}$.
We then have $c_{r}V_{r}\subseteq\sum_{a=1}^{r-1}c_{a}\mathrm{k}\langle
X\rangle$ and $c_{r}\mathrm{k}\langle V_{r}B\rangle\subseteq\sum_{a=1}%
^{r-1}c_{a}\mathrm{k}\langle XB\rangle$. Since $\mathrm{k}\langle
V_{r}B\rangle$ is a subspace of $\mathrm{k}\langle XB\rangle$, this gives
\[
\mathrm{k}\langle CXB\rangle=\sum_{a=1}^{r}c_{a}\mathrm{k}\langle
XB\rangle=\sum_{a=1}^{r-1}c_{a}\mathrm{k}\langle XB\rangle+c_{r}%
\mathrm{k}\langle XB\rangle=\sum_{a=1}^{r-1}c_{a}\mathrm{k}\langle
XB\rangle+c_{r}W
\]
where $W$ is a $\mathrm{k}$-subspace of $\mathrm{k}\langle XB\rangle$ such
that $\mathrm{k}\langle XB\rangle=W\oplus\mathrm{k}\langle V_{r}B\rangle$.\ We
have, in particular, $\dim_{\mathrm{k}}(c_{r}W)=\dim_{\mathrm{k}}%
(W)=\dim_{\mathrm{k}}(XB)-\dim_{\mathrm{k}}(V_{r}B)$. We thus obtain%
\begin{equation}
\dim_{\mathrm{k}}(CXB)\leq\dim_{\mathrm{k}}(\sum_{a=1}^{r-1}c_{a}%
\mathrm{k}\langle XB\rangle)+\dim_{\mathrm{k}}(XB)-\dim_{\mathrm{k}}(V_{r}B).
\label{ineq2}%
\end{equation}
By the induction hypothesis, we have%
\[
\dim_{\mathrm{k}}(\sum_{a=1}^{r-1}c_{a}\mathrm{k}\langle XB\rangle
)=\dim_{\mathrm{k}}(C^{\prime}XB)\leq\rho\dim_{\mathrm{k}}(C^{\prime}%
X)=\rho\dim_{\mathrm{k}}(\sum_{a=1}^{r-1}c_{a}\mathrm{k}\langle X\rangle)
\]
with $C^{\prime}=C\setminus\{c_{r}\}$. Thus by using (\ref{deco}) with
$j=r-1,$ one gets%
\[
\dim_{\mathrm{k}}(\sum_{a=1}^{r-1}c_{a}\mathrm{k}\langle XB\rangle)\leq
\rho\sum_{a=1}^{r-1}\dim_{\mathrm{k}}(c_{a}X_{a}).
\]
Since $V_{r}\subseteq\mathrm{k}\langle X\rangle,$ we have $\dim_{\mathrm{k}%
}(V_{r}B)\geq\rho\dim_{\mathrm{k}}(V_{r})$. By definition of $X$, we have
$\dim_{\mathrm{k}}(XB)=\rho\dim_{\mathrm{k}}(X)$.\ Therefore
\[
\dim_{\mathrm{k}}(XB)-\dim_{\mathrm{k}}(V_{r}B)\leq\rho\dim_{\mathrm{k}%
}(X)-\rho\dim_{\mathrm{k}}(V_{r})\leq\rho\dim_{\mathrm{k}}(X_{r}).
\]
Combining the two previous inequalities with (\ref{ineq2}), we finally obtain%
\begin{multline*}
\dim_{\mathrm{k}}(CXB)\leq\rho\sum_{a=1}^{r-1}\dim_{\mathrm{k}}(c_{a}%
X_{a})+\rho\dim_{\mathrm{k}}(c_{r}X_{r})\leq\\
\rho\sum_{a=1}^{r}\dim_{\mathrm{k}}(c_{a}X_{a})=\rho\dim_{\mathrm{k}}%
(\sum_{a=1}^{r}c_{a}\mathrm{k}\langle X\rangle)=\rho\dim_{\mathrm{k}}(CX).
\end{multline*}
where the second to last equality is obtained by (\ref{deco}) with $j=r$.
\end{proof}

\bigskip

\begin{corollary}
\label{Cor_petri}Let $K$ be a division ring containing the field $\mathrm{k}$
in its center. Consider two finite subsets $A$ and $B$ of $K^{\ast}$. Assume
$\alpha$ is a positive real such that $\dim_{\mathrm{k}}(AB)\leq\alpha
\dim_{\mathrm{k}}(A)$. Then there exists a subset $X\subseteq\mathrm{k}\langle
A\rangle$ such that, for any finite subset $C$ of $K^{\ast}$, $\dim
_{\mathrm{k}}(CXB)\leq\alpha\dim_{\mathrm{k}}(CX)$.
\end{corollary}

\begin{proof}
Let $X\subseteq\mathrm{k}\langle A\rangle$ be such that $\rho=r(X)$.\ We have
\[
\rho=\frac{\dim_{\mathrm{k}}(XB)}{\dim_{\mathrm{k}}(X)}\leq\frac
{\dim_{\mathrm{k}}(AB)}{\dim_{\mathrm{k}}(A)}\leq\alpha
\]
by definition of $\rho$. We then apply Proposition \ref{prop-CXB}, which
yields $\dim_{\mathrm{k}}(CXB)\leq\alpha\dim_{\mathrm{k}}(CX)$.
\end{proof}

\subsection{Pl\"{u}nnecke upper bounds for $\dim_{\mathrm{k}}(AB)$}

We assume in this paragraph that $K$ is a \emph{commutative extension} of
$\mathrm{k}$. The following theorem can be regarded as a linear version of
Theorem \ref{Th_plum}. In fact, it is a linear version of the slightly
stronger result obtained by Petridis where $X$ is the same for any positive
integer $n$.

\begin{theorem}
\label{Th_plunn}Let $A$ and $B$ be nonempty, finite subsets in $K^{\ast}%
$.\ Assume that $\dim_{\mathrm{k}}(AB)\leq\alpha\dim_{\mathrm{k}}(A)$ where
$\alpha$ is a positive real. Then there exists a subset $X\subseteq
\mathrm{k}\langle A\rangle$ such that, for any positive integer $n$,
\[
\dim_{\mathrm{k}}(XB^{n})\leq\alpha^{n}\dim_{\mathrm{k}}(X)\text{.}%
\]
In particular, $\dim_{\mathrm{k}}(A^{2})\leq\alpha\dim_{\mathrm{k}}(A)$
implies that $\dim(A^{n})\leq\alpha^{n}\dim(A)$.
\end{theorem}

\begin{proof}
The proof is by induction on $n$.\ Let $X$ be such that $\rho=r(X)$. For
$n=1$, we have
\[
\dim_{\mathrm{k}}(XB)=\dim_{\mathrm{k}}(X)\rho\leq\dim_{\mathrm{k}}%
(X)\frac{\dim_{\mathrm{k}}(AB)}{\dim_{\mathrm{k}}(A)}\leq\alpha\dim
_{\mathrm{k}}(X).
\]
For any $n>1$, we set $C=B^{n-1}$.\ By applying Proposition \ref{prop-CXB}
with $C=B^{n-1}$, we have
\[
\dim_{\mathrm{k}}(XB^{n})=\dim_{\mathrm{k}}(B^{n-1}XB)\leq\frac{\dim
_{\mathrm{k}}(B^{n-1}X)\dim_{\mathrm{k}}(XB)}{\dim_{\mathrm{k}}(X)}.
\]
Next, by the induction hypothesis, $\dim_{\mathrm{k}}(B^{n-1}X)=\dim
_{\mathrm{k}}(XB^{n-1})\leq\alpha^{n-1}\dim(X)$. Therefore%
\[
\dim_{\mathrm{k}}(XB^{n})\leq\alpha^{n-1}\dim_{\mathrm{k}}(XB)\leq\alpha
^{n}\dim_{\mathrm{k}}(X)
\]
where the last inequality follows from the case $n=1$.
\end{proof}

\bigskip

\noindent\textbf{Remarks:}

\begin{enumerate}
\item Observe that commutativity is crucial in the previous proof.
Nevertheless, the conclusion of Theorem \ref{Th_plunn} remains valid in a
division ring with the additional assumption%
\[
ab=ba\text{ for any }a\in A\text{ and }b\in B\text{.}%
\]
Indeed, we then have $xy=yx$ for any $x\in\mathrm{k}\langle A\rangle$ and
$y\in\mathrm{k}\langle B\rangle$. So in the previous proof we still have
$\mathrm{k}\langle XB^{n}\rangle=\mathrm{k}\langle B^{n-1}XB\rangle$ and
$\mathrm{k}\langle B^{n-1}X\rangle=\mathrm{k}\langle XB^{n-1}\rangle$ since
$X\subset\mathrm{k}\langle A\rangle$.

\item Assertions 1 and 2 of Proposition \ref{prop_link} permits one to recover
the Pl\"{u}nnecke-Ruzsa cardinality estimates for abelian groups.
\end{enumerate}

\subsection{Double and triple product}

We now show how to estimate, in a field $K$, the dimension of the vector span
generated by a triple products set in terms of the dimensions of the vector
spans of the corresponding double product set. This is a linear version of
Theorem 1.9.2 in \cite{Ruz} :

\begin{theorem}
Let $A,B$ and $C$ be finite subsets of a group $G$. Then%
\[
\left\vert ABC\right\vert ^{2}\leq\left\vert AB\right\vert \left\vert
BC\right\vert \max_{b\in B}\left\vert AbC\right\vert .
\]

\end{theorem}

\begin{theorem}
\label{Th_triple}Consider a division ring $K$ containing the field
$\mathrm{k}$ in its center.\ Let $A,B$ and $C$ be finite, nonempty subsets of
$K^{\ast}$. Then%
\begin{equation}
(\dim_{\mathrm{k}}(ABC))^{2}\leq\dim_{\mathrm{k}}(AB)\dim_{\mathrm{k}}%
(BC)\max_{b\in B}\{\dim_{\mathrm{k}}(AbC)\}. \label{triple}%
\end{equation}
In particular, when $K$ is a field, we have%
\[
\dim_{\mathrm{k}}(ABC)^{2}\leq\dim_{\mathrm{k}}(AB)\dim_{\mathrm{k}}%
(BC)\dim_{\mathrm{k}}(AC).
\]

\end{theorem}

\begin{proof}
We proceed by induction on $\left\vert B\right\vert $. When $B=\{b\}$, we
obtain
\[
\dim_{\mathrm{k}}(AbC)^{2}\leq\dim_{\mathrm{k}}(Ab)\dim_{\mathrm{k}}%
(bC)\dim_{\mathrm{k}}(AbC)
\]
by observing that $\dim_{\mathrm{k}}(AbC)\leq\dim_{\mathrm{k}}(Ab)\dim
_{\mathrm{k}}(C)$ and $\dim_{\mathrm{k}}(C)=\dim_{\mathrm{k}}(bC)$.\ Now
assume $\left\vert B\right\vert >1$ and fix $b\in B$ such that $\max_{u\in
B}\{\dim_{\mathrm{k}}(AuC)\}=\dim_{\mathrm{k}}(AbC)$.\ Set $m=\dim
_{\mathrm{k}}(AbC)$.\ Write $B^{\prime}=B\setminus\{b\}$.\ Set $A=\{a_{1}%
,\ldots,a_{r}\}$ and $C=\{c_{1},\ldots,c_{s}\}$. We have $\mathrm{k}\langle
AB\rangle=\mathrm{k}\langle AB^{\prime}\rangle+\sum_{a\in A}\mathrm{k}\langle
ab\rangle$. Let $S_{AB^{\prime}}$ be a basis of $\mathrm{k}\langle AB^{\prime
}\rangle$. Since $S_{A^{\prime}B}\cup Ab$ generates $\mathrm{k}\langle
AB\rangle$, there exists a subset $A^{\flat}$ of $A$ such that
\[
\mathrm{k}\langle AB\rangle=\mathrm{k}\langle AB^{\prime}\rangle
\oplus\bigoplus_{a\in A^{\flat}}\mathrm{k}\langle ab\rangle.
\]
Similarly, there exists a subset $C^{\flat}$ of $C$ such that
\[
\mathrm{k}\langle BC)=\mathrm{k}\langle B^{\prime}C\rangle\oplus
\bigoplus_{c\in C^{\flat}}\mathrm{k}\langle bc\rangle.
\]
We get%
\begin{multline*}
\mathrm{k}\langle ABC\rangle=\mathrm{k}\langle AB^{\prime}C\rangle+\sum_{a\in
A^{\flat}}\mathrm{k}\langle abC\rangle=\mathrm{k}\langle AB^{\prime}%
C\rangle+\sum_{a\in A^{\flat}}\mathrm{k}\langle aBC\rangle=\\
\mathrm{k}\langle AB^{\prime}C\rangle+\sum_{a\in A^{\flat}}\mathrm{k}\langle
aB^{\prime}C\rangle+\sum_{a\in A^{\flat}}\sum_{c\in C^{\flat}}\mathrm{k}%
\langle abc\rangle.
\end{multline*}
But $\sum_{a\in A^{\flat}}\mathrm{k}\langle aB^{\prime}C\rangle\subseteq
\mathrm{k}\langle AB^{\prime}C\rangle$, and thus%
\[
\mathrm{k}\langle ABC\rangle=\mathrm{k}\langle AB^{\prime}C\rangle+\sum_{a\in
A^{\flat}}\sum_{c\in C^{\flat}}\mathrm{k}\langle abc\rangle.
\]
Let $S_{AB^{\prime}C}$ be a basis of $\mathrm{k}\langle AB^{\prime}C\rangle.$
By the previous decomposition, there exists $X\subseteq A^{\flat}\times
B^{\flat}$ such that%
\[
\mathrm{k}\langle ABC\rangle=\mathrm{k}\langle AB^{\prime}C\rangle
\oplus\bigoplus_{(a,c)\in X}\mathrm{k}\langle abc\rangle.
\]
Set $\alpha=\left\vert X\right\vert $, $\beta=\left\vert A^{\flat}\right\vert
$ and $\gamma=\left\vert C^{\flat}\right\vert $.\ We have to prove
(\ref{triple}), that is,
\begin{equation}
(\dim_{\mathrm{k}}(AB^{\prime}C)+\alpha)^{2}\leq(\dim_{\mathrm{k}}(AB^{\prime
})+\beta)(\dim_{\mathrm{k}}(B^{\prime}C)+\gamma)m. \label{tripleref}%
\end{equation}
By the induction hypothesis, we have
\begin{equation}
\dim_{\mathrm{k}}(AB^{\prime}C)^{2}\leq\dim_{\mathrm{k}}(AB^{\prime}%
)\dim_{\mathrm{k}}(B^{\prime}C)m \label{tripleinduc}%
\end{equation}
because $\max_{u\in B^{\prime}}\{\dim_{\mathrm{k}}(AuC)\}\leq\max_{u\in
B}\{\dim_{\mathrm{k}}(AuC)\}=m$. We have $\bigoplus_{(a,c)\in X}%
\mathrm{k}\langle abc\rangle\subseteq\mathrm{k}\langle AbC\rangle$. So
$\alpha\leq m$. Since $X\subseteq A^{\flat}\times B^{\flat}$, we have also
$\alpha\leq\beta\gamma$. We get $\alpha^{2}\leq m\beta\gamma.$ By multiplying
in (\ref{tripleinduc}), this gives
\[
\alpha^{2}\dim_{\mathrm{k}}(AB^{\prime}C)^{2}\leq\beta\gamma\dim_{\mathrm{k}%
}(AB^{\prime})\dim_{\mathrm{k}}(B^{\prime}C)m^{2}.
\]
Therefore%
\[
\alpha\dim_{\mathrm{k}}(AB^{\prime}C)\leq m\sqrt{\beta\gamma\dim_{\mathrm{k}%
}(AB^{\prime})\dim_{\mathrm{k}}(B^{\prime}C)}\leq m\frac{\gamma\dim
_{\mathrm{k}}(AB^{\prime})+\beta\dim_{\mathrm{k}}(B^{\prime}C)}{2}%
\]
by applying the geometric-arithmetic means inequality.\ So
\[
2\alpha\dim_{\mathrm{k}}(AB^{\prime}C)\leq m(\gamma\dim_{\mathrm{k}%
}(AB^{\prime})+\beta\dim_{\mathrm{k}}(B^{\prime}C)).
\]
Combining this last equality with $\alpha^{2}\leq m\beta\gamma$ and
(\ref{tripleinduc}), we finally get%
\[
\dim_{\mathrm{k}}(AB^{\prime}C)^{2}+2\alpha\dim_{\mathrm{k}}(AB^{\prime
}C)+\alpha^{2}\leq m(\dim_{\mathrm{k}}(AB^{\prime})\dim_{\mathrm{k}}%
(B^{\prime}C)+\gamma\dim_{\mathrm{k}}(AB^{\prime})+\beta\dim_{\mathrm{k}%
}(B^{\prime}C)+\beta\gamma)
\]
as desired.
\end{proof}

\bigskip

By using Theorems \ref{Th_plunn} and \ref{Th_triple}, we can obtain a bound
for $\dim_{\mathrm{k}}(A^{3})$ knowing $\dim_{\mathrm{k}}(A^{2})$ and
$\dim_{\mathrm{k}}(A)$.

\begin{corollary}
Consider a field extension $K$ of $\mathrm{k}$ and $A$ a nonempty, finite
subset of $K^{\ast}$.$\;$Assume $\dim_{\mathrm{k}}(A)=m$ and $\dim
_{\mathrm{k}}(A^{2})=n$. Then,
\[
\dim_{\mathrm{k}}(A^{3})\leq\min(n^{3/2},\frac{n^{3}}{m^{2}}).
\]

\end{corollary}

\section{Pl\"{u}nnecke-type estimates in associative algebras}

\label{Sec_alge}

The arguments we have used in the proofs of Section \ref{Sec_Plum_K} to obtain
Pl\"{u}nnecke-type estimates in division rings in fact remain valid in the
more general context of associative unital algebras with suitable hypotheses
on the subspaces considered. More precisely, let $\mathcal{A}$ be a unital
associative algebra over the field $\mathrm{k}$.\ Write $U(\mathcal{A)}$ for
the group of invertible elements in $\mathcal{A}$. As classical examples, we
can consider any matrix algebra containing the identity matrix or the group algebras.

Given a subset $A$ of $\mathcal{A}$ and $x\in\mathcal{A}$, $\dim_{\mathrm{k}%
}(xA)$ and $\dim_{\mathrm{k}}(Ax)$ do not necessarily coincide with
$\dim_{\mathrm{k}}(A)$. This is nevertheless true when $x\in U(\mathcal{A})$.
Let $A,B$ and $C$ be nonempty, finite subsets of $\mathcal{A}$ such that

\begin{itemize}
\item $B\cap U(\mathcal{A})\neq\emptyset,$

\item $C\subseteq U(\mathcal{A)}$.
\end{itemize}

Then for any $\mathrm{k}$-subspace $V\neq\{0\}$ of $\mathrm{k}\langle
A\rangle$,%
\[
\dim_{\mathrm{k}}(VB)\geq\dim_{\mathrm{k}}(V),\quad r(V)=\frac{\dim
_{\mathrm{k}}(VB)}{\dim_{\mathrm{k}}(V)}>0\text{ and }\rho:=\min
_{V\subseteq\mathrm{k}\langle A\rangle,V\neq\{0\}}r(V)>0.
\]
There thus also exists a nonempty set $\mathrm{k}\langle X\rangle\subseteq A$
such that $r(\mathrm{k}\langle X\rangle)=\rho$. One then easily verifies that
the arguments used in the proof of Proposition \ref{prop-CXB} remain valid for
the algebra $\mathcal{A}$ with the previous assumptions on $A,B$ and $C$. We
then obtain the following statements which generalize Corollary
\ref{Cor_petri}, Theorem \ref{Th_plunn} and Theorem \ref{Th_triple}.

\begin{corollary}
\label{CorAlg}Consider $A$ and $B$ two finite subsets of $\mathcal{A}$ with
$B\cap U(\mathcal{A})\neq\emptyset$.\ Assume $\alpha$ is a positive real such
that $\dim_{\mathrm{k}}(AB)\leq\alpha\dim_{\mathrm{k}}(A)$. Then there exists
a subset $X\subseteq\mathrm{k}\langle A\rangle$ such that, for any finite
subset $C$ of $U(\mathcal{A})$, $\dim_{\mathrm{k}}(CXB)\leq\alpha
\dim_{\mathrm{k}}(CX)$.
\end{corollary}

\begin{theorem}
\label{ThAlg1}Assume $\mathcal{A}$ is commutative.\ Let $A$ be a nonempty
finite subset of $\mathcal{A}$ and $B$ be a nonempty, finite subset of
$U(\mathcal{A})$.\ Assume that $\dim_{\mathrm{k}}(AB)\leq\alpha\dim
_{\mathrm{k}}(A)$ where $\alpha$ is a positive real. Then there exists a
subset $X\subseteq\mathrm{k}\langle A\rangle$ such that, for any positive
integer $n$,
\[
\dim_{\mathrm{k}}(XB^{n})\leq\alpha^{n}\dim_{\mathrm{k}}(X)\text{.}%
\]
In particular, if $A\subseteq U(\mathcal{A})$, $\dim_{\mathrm{k}}(A^{2}%
)\leq\alpha\dim_{\mathrm{k}}(A)$ implies that $\dim(A^{n})\leq\alpha^{n}%
\dim(A)$.
\end{theorem}

\begin{theorem}
\label{ThAlg}Let $A,B$ and $C$ be finite nonempty subsets of $\mathcal{A}$
such that $B\subseteq U(\mathcal{A})$. Then%
\[
\dim_{\mathrm{k}}(ABC)^{2}\leq\dim_{\mathrm{k}}(AB)\dim_{\mathrm{k}}%
(BC)\max_{b\in B}\{\dim_{\mathrm{k}}(AbC)\}.
\]
In particular, when $\mathcal{A}$ is commutative, we have%
\[
\dim_{\mathrm{k}}(ABC)^{2}\leq\dim_{\mathrm{k}}(AB)\dim_{\mathrm{k}}%
(BC)\dim_{\mathrm{k}}(AC).
\]

\end{theorem}

\noindent\textbf{Remarks:}

\begin{enumerate}
\item When $\mathcal{A}$ is not commutative, the conclusion of Theorem
\ref{ThAlg1} remains valid if we assume that $ab=ba$ for any $a\in A$ and
$b\in B$.

\item The proof of the Kneser type theorems we have obtained in Section
\ref{Sec_Kneser} (and also that of the linear Olson theorem) cannot be so
easily adapted to the case of associative algebras. Indeed, given subsets $A$
and $B$ of $U(\mathcal{A})$, $\mathrm{k}\langle A\rangle\cap\mathrm{k}\langle
B\rangle$ may have an empty intersection with $U(\mathcal{A})$.\ In
particular, arguments based on the use of linear versions of the Dyson or
Kemperman transform fail.

\item With the notation of \S \ \ref{subsec-linknonabelina}, consider $G$ a
group and $X,Y,Z$ finite subsets in $G$. Observe that $A_{X},A_{Y}$ and
$A_{Z}$ are subsets of $U(\mathbb{C}[G]).$ By Theorem \ref{ThAlg}, we recover
Ruzsa's inequality%
\[
\left\vert ABC\right\vert ^{2}\leq\left\vert XY\right\vert \left\vert
YZ\right\vert \max_{y\in Y}\{\left\vert XyZ\right\vert \}.
\]

\item Theorem \ref{ThAlg1} also yields Theorem \ref{Th_plum} also by
considering subsets of $\mathbb{C}[G]$.
\end{enumerate}

\section{Kneser type theorems for division rings}

\label{Sec_Kneser}In this section, $K$ is a division ring and $\mathrm{k}$ a
field contained in the center of $K$.

\subsection{Assuming $A$ is commutative}

Consider a finite nonempty subset $A$ of $K^{\ast}$. We say that $A$ is
commutative when $aa^{\prime}=a^{\prime}a$ for any $a,a^{\prime}\in A$. This
then implies that the elements of $\mathrm{k}\langle A\rangle$ are pairwise
commutative.\ Moreover the division ring $\mathbb{D}(A)$ generated by $A$ is a
field. Typical examples of commutative sets are geometric progressions
$A=\{a^{r},a^{r+1},\ldots,a^{r+s}\}$ with $r$ and $s$ integers. The following
theorem is the linearization of a theorem by Diderrich \cite{Die} extending
Kneser's theorem for arbitrary groups when only the subset $A$ is assumed
commutative. It was shown by Hamidoune in \cite{Hami4} that Diderrich's result
can also be derived from the original Kneser theorem in abelian group.\ Here
we will prove our theorem without using Theorem \ref{Th_HL}. Also we will
assume that $\mathrm{k}$ is infinite.\ When $K$ is finite-dimensional over
$\mathrm{k}$ and $\mathrm{k}$ is finite, $K$ is a field, so we can apply
Theorem \ref{Th_HL}.

\begin{theorem}
\label{Th_KL_Acom}Assume $\mathrm{k}$ is infinite and every algebraic element
of $K$ is separable over $\mathrm{k}$.

\begin{enumerate}
\item Let $A$ and $B$ be two finite nonempty subsets of $K^{\ast}$ such that
$A$ is commutative. Then either $\dim_{\mathrm{k}}(AB)\geq\dim_{\mathrm{k}%
}(A)+\dim_{\mathrm{k}}(B)-1$ or $\mathrm{k}\langle AB\rangle$ is left periodic.

\item Let $A_{1},\ldots,A_{n}$ be a collection of finite nonempty subsets of
$K^{\ast}$ such that $A_{1},\ldots,A_{n-1}$ are commutative. Then either
$\dim_{\mathrm{k}}(A_{1}\cdots A_{n})\geq\sum_{i=1}^{n}\dim_{\mathrm{k}}%
(A_{i})-(n-1)$ or $\mathrm{k}\langle A_{1}\cdots A_{n}\rangle$ is periodic.
\end{enumerate}
\end{theorem}

\noindent\textbf{Remark: }In the linear Olson theorem established in
\cite{EL2} (where $A$ is not assumed commutative), we only obtain that
$\dim_{\mathrm{k}}(AB)\geq\dim_{\mathrm{k}}(A)+\dim_{\mathrm{k}}(B)-1$, or
$\mathrm{k}\langle AB\rangle$ contains a (left or right) periodic subspace.

\bigskip

To prove the theorem, we need to adapt the arguments of \cite{Xian} to our
noncommutative situation. We begin with the following lemma based on the
linear Dyson transform.

\begin{lemma}
\label{Lem_diff}Let $A$ and $B$ be two finite, nonempty subsets of $K^{\ast}$
such that $A$ is commutative. Then, for each nonzero $a\in\mathrm{k}\langle
A\rangle$, there exists a (commutative) subfield $H_{a}$ of $K$ such that
$\mathrm{k}\subseteq H_{a}\subseteq\mathbb{D}(A)$ and a vector space
$V_{a}\neq\{0\}$ contained in $\mathrm{k}\langle AB\rangle$ such that
$H_{a}V_{a}=V_{a},$ $\mathrm{k}\langle aB\rangle\subseteq V_{a}$ and
\[
\dim_{\mathrm{k}}(V_{a})+\dim_{\mathrm{k}}(H_{a})\geq\dim_{\mathrm{k}}%
(A)+\dim_{\mathrm{k}}(B).
\]

\end{lemma}

\begin{proof}
By replacing $A$ by $A^{\prime}=a^{-1}A$, we can assume $a=1$. Indeed, if
there exist a subfield $H\subseteq\mathbb{D}(A^{\prime})$ and a vector space
$V\neq\{0\}$ contained in $\mathrm{k}\langle A^{\prime}B\rangle$ such that
$HV=V$ and $\mathrm{k}\langle B\rangle\subseteq V$ with
\[
\dim_{\mathrm{k}}(V)+\dim_{\mathrm{k}}(H)\geq\dim_{\mathrm{k}}(A^{\prime
})+\dim_{\mathrm{k}}(B),
\]
it suffices to take $V_{a}=aV$ and $H_{a}=H\subseteq\mathbb{D}(A^{\prime
})\subseteq\mathbb{D}(A)$.\ Since $H\subseteq\mathbb{D}(A),$ we must have
$Ha=aH$ for any $a\in A$ and $H(V_{a})=H(aV)=aHV=aV=V_{a}$. Moreover
$\mathrm{k}\langle aB\rangle=a\mathrm{k}\langle B\rangle\subseteq aV=V_{a}$
and $\dim_{\mathrm{k}}(V_{a})+\dim_{\mathrm{k}}(H_{a})\geq\dim_{\mathrm{k}%
}(A)+\dim_{\mathrm{k}}(B)$ because $\dim_{\mathrm{k}}(V_{a})=\dim_{\mathrm{k}%
}(V)$ and $H_{a}=H$.

We can also assume that $1\in B$ by replacing $B$ by $B^{\prime}=Bb^{-1}$.
Indeed, if there exist a subfield $H^{\prime}\subseteq\mathbb{D}(A)$ and a
vector space $V^{\prime}\neq\{0\}$ contained in $\mathrm{k}\langle AB^{\prime
}\rangle$ such that $H^{\prime}V^{\prime}=V^{\prime}$ and $\mathrm{k}\langle
B^{\prime}\rangle\subseteq V^{\prime}$ with
\[
\dim_{\mathrm{k}}(V^{\prime})+\dim_{\mathrm{k}}(H^{\prime})\geq\dim
_{\mathrm{k}}(A)+\dim_{\mathrm{k}}(B^{\prime}),
\]
it suffices to take $V=V^{\prime}b$ and $H=H^{\prime}$.\ We will have then
$V=V^{\prime}b\subseteq\mathrm{k}\langle AB^{\prime}\rangle b=\mathrm{k}%
\langle AB\rangle,$ $HV=H(V^{\prime}b)=(HV^{\prime})b=V^{\prime}b=V,$
$\mathrm{k}\langle B\rangle=\mathrm{k}\langle B^{\prime}\rangle b\subseteq
V^{\prime}b=V$ and
\[
\dim_{\mathrm{k}}(V)+\dim_{\mathrm{k}}(H)\geq\dim_{\mathrm{k}}(A)+\dim
_{\mathrm{k}}(B)
\]
since $\dim_{\mathrm{k}}(B)=\dim_{\mathrm{k}}(B^{\prime})$ and $\dim
_{\mathrm{k}}(V)=\dim_{\mathrm{k}}(V^{\prime})$.

We thus assume in the remainder of the proof that $1\in A\cap B$ and proceed
by induction on $\dim_{\mathrm{k}}(A).$ When dim$_{\mathrm{k}}(A)=1$, we have
$\mathrm{k}\langle A\rangle=\mathrm{k}=\mathbb{D}(A)$. It suffices to take
$V_{1}=V=\mathrm{k}\langle B\rangle\neq\{0\}$ and $H_{a}=H=\mathrm{k}%
=\mathbb{D}(A)$. Assume $\dim_{\mathrm{k}}(A)>1$.\ Given $e\in\mathrm{k}%
\langle B\rangle$ such that $e\neq0$, define $A(e)$ and $B(e)$ to be finite
subsets of $K^{\ast}$ such that
\[
\mathrm{k}\langle A(e)\rangle=\mathrm{k}\langle A\rangle\cap\mathrm{k}\langle
B\rangle e^{-1}\text{ and }\mathrm{k}\langle B(e)\rangle=\mathrm{k}\langle
B\rangle+\mathrm{k}\langle A\rangle e.
\]
Observe that $\mathrm{k}\langle A(e)\rangle$ and $\mathrm{k}\langle
B(e)\rangle$ contain $\mathrm{k}$ since $1\in A\cap B$. Thus we may and do
assume that $1\in A(e)\cap B(e)$.\ Moreover, $\mathrm{k}\langle A(e)\rangle
\mathrm{k}\langle B(e)\rangle$ is contained in $\mathrm{k}\langle AB\rangle$.
Indeed, for $v\in\mathrm{k}\langle A\rangle\cap\mathrm{k}\langle B\rangle
e^{-1}$ and $w\in\mathrm{k}\langle B\rangle,$ we have $vw\in\mathrm{k}\langle
A\rangle\mathrm{k}\langle B\rangle\subseteq\mathrm{k}\langle AB\rangle$
because $v\in\mathrm{k}\langle A\rangle$. Set $v=ze^{-1}$ with $z\in
\mathrm{k}\langle B\rangle$. If $w\in\mathrm{k}\langle A\rangle e,$ we have
$vw\in ze^{-1}\mathrm{k}\langle A\rangle e$. But $ze^{-1}\in\mathrm{k}\langle
A\rangle$ and $A$ is commutative. Therefore, $vw\in\mathrm{k}\langle A\rangle
ze^{-1}e=\mathrm{k}\langle A\rangle z\subseteq\mathrm{k}\langle A\rangle
\mathrm{k}\langle B\rangle\subseteq\mathrm{k}\langle AB\rangle$.\ In
particular, $\dim_{\mathrm{k}}(A(e)B(e))\leq\dim_{\mathrm{k}}(AB)$.\ We get%
\begin{multline*}
\dim_{\mathrm{k}}(A(e))+\dim_{\mathrm{k}}(B(e))=\dim_{\mathrm{k}}%
(\mathrm{k}\langle A\rangle\cap\mathrm{k}\langle B\rangle e^{-1}%
)+\dim_{\mathrm{k}}(\mathrm{k}\langle B\rangle+\mathrm{k}\langle A\rangle
e)=\\
\dim_{\mathrm{k}}(\mathrm{k}\langle A\rangle e\cap\mathrm{k}\langle
B\rangle)+\dim_{\mathrm{k}}(\mathrm{k}\langle B\rangle+\mathrm{k}\langle
A\rangle e)=\dim_{\mathrm{k}}(Ae)+\dim_{\mathrm{k}}(B)=\dim_{\mathrm{k}%
}(A)+\dim_{\mathrm{k}}(B).
\end{multline*}
Also $A(e)\subseteq\mathrm{k}\langle A\rangle$.\ 

Assume $\mathrm{k}\langle A(e)\rangle=\mathrm{k}\langle A\rangle$ for any
nonzero $e\in\mathrm{k}\langle B\rangle$.\ Then $\mathrm{k}\langle A\rangle
e\subseteq\mathrm{k}\langle B\rangle$ for any nonzero $e\in\mathrm{k}\langle
B\rangle$. Thus $\mathrm{k}\langle AB\rangle\subseteq\mathrm{k}\langle
B\rangle$.\ Since $1\in A$, we have in fact $\mathrm{k}\langle AB\rangle
=\mathrm{k}\langle B\rangle$.\ The sub division ring $H=\mathbb{D}(A)$ is a
field since $A$ is commutative and it contains $\mathrm{k}$ since $1\in A$.
Take $V=\mathrm{k}\langle B\rangle\neq\{0\}$. Then $HV=V$ since $AV=V$.\ We
clearly have $V=\mathrm{k}\langle AB\rangle$ and $B\subseteq V$ as desired.

Now assume $\mathrm{k}\langle A(e)\rangle\neq\mathrm{k}\langle A\rangle$ for
at least one nonzero $e\in\mathrm{k}\langle A\rangle$.\ Then $0<\dim
_{\mathrm{k}}(A(e))<\dim_{\mathrm{k}}(A)$ and $1\in A(e)\cap B(e)$.\ By our
induction hypothesis, there exist a subfield $H$ of $\mathbb{D}(A(e))\subseteq
\mathbb{D}(A)$ containing $\mathrm{k}$ and a nonzero $\mathrm{k}$-vector space
$V\subseteq\mathrm{k}\langle A(e)B(e)\rangle\subseteq\mathrm{k}\langle
AB\rangle$, such that $HV=V$ and $\mathrm{k}\langle B\rangle\subseteq
\mathrm{k}\langle B(e)\rangle\subseteq V$ with
\[
\dim_{\mathrm{k}}(V)+\dim_{\mathrm{k}}(H)\geq\dim_{\mathrm{k}}(A(e))+\dim
_{\mathrm{k}}(B(e))=\dim_{\mathrm{k}}(A)+\dim_{\mathrm{k}}(B).
\]
The subfield $H\subseteq\mathbb{D}(A)$ and the nonzero space $V\supset
\mathrm{k}\langle B\rangle$ satisfy the statement of the lemma for the pair of
subsets $A$ and $B$ which concludes the proof.
\end{proof}

\bigskip

As in the proof of Theorem \ref{Th_HL}, we also need the following lemma,
which is an application of the Vandermonde determinant.

\begin{lemma}
\label{Lemma_VDM}Let $V$ be a $n$-dimensional vector space over the infinite
field $\mathrm{k}$.\ Assume $x_{1},\ldots,x_{n}$ form a basis of $V$ over
$\mathrm{k}$. Then any $n$ vectors in the set
\[
\{x_{1}+\alpha x_{2}+\cdots+\alpha^{n-1}x_{n}\mid\alpha\in\mathrm{k}\}
\]
form a basis of $V$ over $\mathrm{k}$.
\end{lemma}

The proof of Theorem \ref{Th_KL_Acom} requires a last lemma. Consider a field
extension $F$ over the infinite field $\mathrm{k}$. Let $y_{1},\ldots,y_{r}$
be algebraically independent indeterminates. Set $\mathrm{k}^{\prime
}=\mathrm{k}(y_{1},\ldots,y_{r})$ and $F^{\prime}=F(y_{1},\ldots,y_{r})$.

\begin{lemma}
\label{Lem_tranfert}Assume $\dim_{\mathrm{k}^{\prime}}(F^{\prime})$ is finite.
Then $\dim_{\mathrm{k}}(F)=\dim_{\mathrm{k}^{\prime}}(F^{\prime})$.
\end{lemma}

\begin{proof}
Let $b_{1},\ldots,b_{n}$ be elements in $F$ linearly independent over
$\mathrm{k}$. Assume there exists nonzero elements $P_{1},\ldots,P_{n}$ in
$\mathrm{k}^{\prime}$ such that%
\[
\sum_{i=1}^{n}P_{i}b_{i}=0\text{.}%
\]
By multiplying by an overall polynomial of $\mathrm{k}^{\prime}$, we can
assume that the $P_{i}$'s belong to $k[y_{1},\ldots,y_{r}]$. Set
$I=\{i\in\{1,\ldots n\}\mid P_{i}\neq0\}$. Assume $I$ is nonempty. Set
$P=\prod_{i\in I}P_{i}$ and consider $V=\{(u_{1},\ldots,u_{r})\in
\mathrm{k}^{r}\mid P(u_{1},\ldots,u_{r})=0\}$. Since $P\neq0$ and $\mathrm{k}$
is infinite, $V\neq k^{n}$. Therefore, there exists $(u_{1},\ldots,u_{r})$ in
$\mathrm{k}^{r}$ such that $P(u_{1},\ldots,u_{r})\neq0$. Then for any $i\in
I$, we have $P_{i}(u_{1},\ldots,u_{r})\neq0$ and
\[
\sum_{i=1}^{n}P_{i}(u_{1},\ldots,u_{r})b_{i}=0
\]
which contradicts our assumption that $b_{1},\ldots,b_{n}$ are linearly
independent over $\mathrm{k}$. We have proved that $\dim_{\mathrm{k}}%
(F)\leq\dim_{\mathrm{k}^{\prime}}(F^{\prime})$. In particular, $\dim
_{\mathrm{k}}(F)$ is finite. We assume from now on that $\dim_{\mathrm{k}%
}(F)=n$ and $\{b_{1},\ldots,b_{n}\}$ is a basis of $F$ over $\mathrm{k}$. Let%
\[
E=%
%TCIMACRO{\tbigoplus \limits_{i=1}^{n}}%
%BeginExpansion
{\textstyle\bigoplus\limits_{i=1}^{n}}
%EndExpansion
\mathrm{k}^{\prime}b_{i}\text{.}%
\]
Clearly $E$ is a $n$-dimensional $\mathrm{k}^{\prime}$-space closed under
multiplication. So by Lemma \ref{lemma_stab}, $E$ is a subfield of $F^{\prime
}$. Moreover $E$ contains the polynomial ring $F[y_{1},\ldots,y_{r}]$ since
each element of $F$ decomposes as a $\mathrm{k}$-linear combination of the
$b_{i}$'s. This implies that $E=F^{\prime}$ and $\dim_{\mathrm{k}}(F)\geq
\dim_{\mathrm{k}^{\prime}}(F^{\prime})$. So $\dim_{\mathrm{k}}(F)=\dim
_{\mathrm{k}^{\prime}}(F^{\prime})$ as desired.
\end{proof}

\bigskip

We are now ready to prove \ref{Th_KL_Acom}.

\begin{proof}
(of Theorem \ref{Th_KL_Acom})

1: Let $\mathcal{B}=\{x_{1},\ldots,x_{n}\}$ be a basis of $\mathrm{k}\langle
A\rangle$.\ For any $\alpha\in\mathrm{k}$, set $x_{\alpha}=x_{1}+\alpha
x_{2}+\cdots+\alpha^{n-1}x_{n}$.\ Observe that $x_{\alpha}\neq0$. Since
$\mathrm{k}$ is infinite and by Lemma \ref{Lem_diff}, there exist a subfield
$H_{\alpha}$ such that $\mathrm{k}\subseteq H_{\alpha}\subseteq\mathbb{D}%
(A)\subseteq K$ and a $\mathrm{k}$-vector space $V_{\alpha}\subseteq
\mathrm{k}\langle AB\rangle$ with $x_{\alpha}B\subseteq V_{\alpha},$
$H_{\alpha}V_{\alpha}=V_{\alpha}$ and $\dim_{\mathrm{k}}(V_{\alpha}%
)+\dim_{\mathrm{k}}(H_{\alpha})\geq\dim_{\mathrm{k}}(A)+\dim_{\mathrm{k}}%
(B)$.\ Since $V_{\alpha}\neq\{0\}$ and $H_{\alpha}$ stabilizes $V_{\alpha
}\subseteq\mathrm{k}\langle AB\rangle$, there exists a nonzero vector $v\in
V_{\alpha}$ such that $H_{\alpha}v\subseteq\mathrm{k}\langle AB\rangle
$.\ Hence $H_{\alpha}\subseteq v^{-1}\mathrm{k}\langle AB\rangle$ and
$\dim_{\mathrm{k}}(H_{\alpha})$ is finite.\ Therefore $H_{\alpha}%
\subseteq\mathbb{D}(A)$ is a finite field extension of $\mathrm{k}$. Let $F$
be the algebraic closure of $\mathrm{k}$ in $\mathbb{D}(A)$.\ The elements of
$H_{\alpha}$ belong to $\mathbb{D}(A)$ and are algebraic over $\mathrm{k}$
since $\dim_{\mathrm{k}}(H_{\alpha})$ is finite. Therefore $H_{\alpha
}\subseteq F$ for any $\alpha\in\mathrm{k}$.

The field $\mathbb{D}(A)$ is finitely generated by $x_{1},\ldots,x_{n}$.
Therefore, if $F=\mathbb{D}(A)$, each $x_{i}$ is algebraic over $\mathrm{k}$
and $\dim_{\mathrm{k}}(F)$ is finite. If $F\subsetneqq\mathbb{D}(A),$ we can
choose a family $y_{1},\ldots,y_{r}$ in $\mathbb{D}(A)$ such that
$\mathrm{k}^{\prime}=\mathrm{k}(y_{1},\ldots,y_{r})$ is purely transcendental
over $\mathrm{k}$ and $\mathbb{D}(A)$ is algebraic finitely generated over
$\mathrm{k}^{\prime}$. Then $\dim_{\mathrm{k}^{\prime}}(\mathbb{D}(A))$ is
finite. Thus by Lemma \ref{Lem_tranfert} $\dim_{\mathrm{k}}(F)=\dim
_{\mathrm{k}^{\prime}}(F(y_{1},\ldots,y_{r}))\leq\dim_{\mathrm{k}^{\prime}%
}(\mathbb{D}(A))$ is finite. So in both cases, we obtain that $\dim
_{\mathrm{k}}(F)$ is finite.

By the separability hypothesis, we obtain that the extension $F$ is separable
over $\mathrm{k}$. Thus, it only admits a finite number of intermediate
extensions. There should exist $n$ distinct elements $\alpha_{1},\ldots
,\alpha_{n}$ in $\mathrm{k}$ such that
\[
H_{\alpha_{1}}=H_{\alpha_{2}}=\cdots=H_{\alpha_{n}}=H.
\]
By Lemma \ref{Lemma_VDM}, $x_{\alpha_{1}},\ldots,x_{\alpha_{n}}$ form a basis
of $\mathrm{k}\langle A\rangle$ over $\mathrm{k}$.\ We thus have
$\mathrm{k}\langle AB\rangle=\sum_{i=1}^{n}x_{\alpha_{i}}\mathrm{k}\langle
B\rangle\subseteq\sum_{i=1}^{n}V_{\alpha_{i}}$ since $x_{\alpha_{i}}%
\mathrm{k}\langle B\rangle\subseteq V_{\alpha_{i}}$ for any $i=1,\ldots
,n$.\ On the other hand, $V_{\alpha_{i}}\subseteq\mathrm{k}\langle AB\rangle$
for any $i=1,\ldots,n$. Hence $\mathrm{k}\langle AB\rangle=\sum_{i=1}%
^{n}V_{\alpha_{i}}$ is stabilized by $H$. If $\mathrm{k}\subsetneqq H$, then
$\mathrm{k}\langle AB\rangle$ is periodic and we are done. Otherwise,
$\mathrm{k}=H$ and we have%
\[
\dim_{\mathrm{k}}(A)+\dim_{\mathrm{k}}(B)\leq\dim_{\mathrm{k}}(V_{\alpha_{i}%
})+1
\]
for any $i=1,\ldots,n$. Since $V_{\alpha_{i}}\subseteq\mathrm{k}\langle
AB\rangle,$ we obtain%
\[
\dim_{\mathrm{k}}(A)+\dim_{\mathrm{k}}(B)\leq\dim_{\mathrm{k}}(AB)+1
\]
as desired.

2: In part (II) of the proof of Theorem \ref{Th_Lin_Kne_n}, we do not use any
commutativity hypothesis on $K$. So both assertions of Theorem
\ref{Th_KL_Acom} are equivalent by exactly the same arguments.
\end{proof}

\bigskip

\noindent\textbf{Remarks:}

\begin{enumerate}
\item When $B$ is assumed commutative, we have a similar statement by
replacing left periodicity by right periodicity.

\item Also observe that the separability hypothesis is always satisfied when
$\mathrm{k}$ has characteristic $0$.

\item Theorem \ref{Th_KL_Acom} means that when $A$ is commutative and
$\dim_{\mathrm{k}}(AB)\leq\dim_{\mathrm{k}}(A)+\dim_{\mathrm{k}}(B)-2$,
$\mathrm{k}\langle AB\rangle$ is an left $H$-module.\ When $A=B$ or $A^{-1}%
=B$, this suggests that spaces $\mathrm{k}\langle A\rangle$ with
$\dim_{\mathrm{k}}(A^{2})=O(\dim A)$ should have interesting properties
related to some $H$-modules of $K$ where $H$ is a subdivision ring of $K$. We
will make this observation precise in the following paragraphs.

\bigskip
\end{enumerate}

Theorem \ref{Th_KL_Acom} also permits one to construct $\mathrm{k}$-subspaces
in $K$ containing subdivision rings. Assume $\dim_{\mathrm{k}}(K)$ is finite
and let $a_{1},\ldots,a_{n}$ be a sequence of elements in $K^{\ast}$
\emph{distinct from }$1$ (with repetition allowed).\ For any nonempty subset
$S\subseteq\{1,\ldots,n\}$, write $a_{S}:=\prod_{i\in S}a_{i}$. Denote by $V$
the $\mathrm{k}$-subspace in $K$ generated by the elements $a_{S}$ when $S$
runs over the nonempty subsets of $\{1,\ldots,n\}$.

\begin{corollary}
Assume $\mathrm{k}$ is infinite, $\dim_{\mathrm{k}}(K)=n>1$ and every element
of $K$ is separable over $\mathrm{k}$.\ Then, with the above notation, the
space $V$ contains a sub division ring $H\varsupsetneqq\mathrm{k}$.
\end{corollary}

\begin{proof}
For any $i=1,\ldots,n,$ set $A_{i}=\{1,a_{i}\}$ and $V^{\prime}=\mathrm{k}%
\langle A_{1}\cdots A_{n}\rangle$. Write $p_{i}=a_{1}\cdots a_{i}$ for any
$i=1,\ldots,n$. If $p_{1},\ldots,p_{n}$ are linearly independent, $V=K$ since
$\dim_{\mathrm{k}}(K)=n$. In particular, $1\in V$.

Now assume $p_{1},\ldots,p_{n}$ are not linearly independent. There exist
$i_{0}\in\{1,\ldots,n\}$ and elements $\alpha_{i},i=i_{0},\ldots,n$, in
$\mathrm{k}$ such that
\[
\sum_{i=i_{0}}^{n}\alpha_{i}p_{i}=0\text{ and }\alpha_{i_{0}}\neq0.
\]
Dividing by $\alpha_{i_{0}}p_{i_{0}},$ we obtain%
\[
1=-\sum_{i=i_{0}+1}^{n}\frac{\alpha_{i}}{\alpha_{i_{0}}}a_{i_{0}+1}\cdots
a_{i}.
\]
We thus also obtain that $1\in V$.

We have proved that $1\in V$. This implies that $V^{\prime}=V$.\ If $V$ is
left periodic, there exists a sub-division ring $H\subseteq K$ such that
$H\varsupsetneqq\mathrm{k}$ and $HV\subseteq V$. Since $1\in V$, we have
$H\subseteq V$ as desired. We can thus assume that $V=\mathrm{k}\langle
A_{1}\cdots A_{n}\rangle$ is not periodic. As the sets $A_{i},i=1,\ldots n$
are commutative, we can apply 2 of Theorem \ref{Th_Lin_Kne_n} which gives
\[
\dim_{\mathrm{k}}(A_{1}\cdots A_{n})\geq\sum_{i=1}^{n}\dim_{\mathrm{k}}%
(A_{i})-(n-1)\geq2n-(n-1)\geq n+1.
\]
We thus obtain a contradiction with the hypothesis $\dim_{\mathrm{k}}(K)=n$.
\end{proof}

\subsection{Linear hamidoune's connectivity}

The notion of connectivity for a subset $S$ of a group $G$ was developed by
Hamidoune in \cite{Hami2}.\ As suggested by Tao in \cite{Tao}, it is
interesting to generalize Hamidoune's definition by introducing an additional
parameter $\lambda$. The purpose of this paragraph is to define a natural
linear version of this connectivity used in \cite{Tao} suited for the
$\mathrm{k}$-subspaces $V$ in $K$, where $K$ is a division ring containing
$\mathrm{k}$ in its center.\ Assume $V$ is a finite-dimensional fixed
$\mathrm{k}$-subspace of $K$ and $\lambda$ is a real parameter. For any
finite-dimensional $\mathrm{k}$-subspace $W$ of $K$, we define%
\begin{equation}
c(W):=\dim_{\mathrm{k}}(WV)-\lambda\dim_{\mathrm{k}}(W). \label{connectivity}%
\end{equation}
For any $x\in K^{\ast}$, we have immediately that $c(xW)=c(W).\;$

\begin{lemma}
For any finite-dimensional subspaces $W_{1},W_{2}$ and $V$ of $K$, we have%
\[
c(W_{1}+W_{2})+c(W_{1}\cap W_{2})\leq c(W_{1})+c(W_{2}).
\]

\end{lemma}

\begin{proof}
We have
\begin{equation}
\dim_{\mathrm{k}}(W_{1}+W_{2})+\dim_{\mathrm{k}}(W_{1}\cap W_{2}%
)=\dim_{\mathrm{k}}(W_{1})+\dim_{\mathrm{k}}(W_{2}) \label{dimform}%
\end{equation}
and%
\[
\dim_{\mathrm{k}}(\mathrm{k}\langle W_{1}V\rangle+\mathrm{k}\langle
W_{2}V\rangle)+\dim_{\mathrm{k}}(\mathrm{k}\langle W_{1}V\rangle\cap
\mathrm{k}\langle W_{2}V\rangle)=\dim_{\mathrm{k}}(W_{1}V)+\dim_{\mathrm{k}%
}(W_{2}V).
\]
Observe that $\mathrm{k}\langle(W_{1}+W_{2})\cdot V\rangle=\mathrm{k}\langle
W_{1}V\rangle+\mathrm{k}\langle W_{2}V\rangle$ and $\mathrm{k}\langle
(W_{1}\cap W_{2})\cdot V\rangle\subseteq\mathrm{k}\langle W_{1}V\rangle
\cap\mathrm{k}\langle W_{2}V\rangle$.\ This gives%
\begin{equation}
\dim_{\mathrm{k}}((W_{1}+W_{2})\cdot V)+\dim_{\mathrm{k}}(W_{1}\cap
W_{2})\cdot V)\leq\dim_{\mathrm{k}}(W_{1}V)+\dim_{\mathrm{k}}(W_{2}V).
\label{ineq}%
\end{equation}
We then obtain the desired equality by subtracting from (\ref{ineq}),
$\lambda$ copies of (\ref{dimform}).
\end{proof}

\bigskip

Similarly to \cite{Hami2}, we define the \emph{connectivity} $\kappa
=\kappa(V)$ as the infimum of $c(W)$ over all finite-dimensional, nonzero
$\mathrm{k}$-subspaces of $K$.\ A \emph{fragment} of $V$ is a
finite-dimensional $\mathrm{k}$-subspace of $K$ which attains the infimum
$\kappa$. An \emph{atom} of $V$ is a fragment of minimal dimension. Since
$c(xW)=c(W)$, any left translate of a fragment is a fragment and any left
translate of an atom is an atom. Since $\dim_{\mathrm{k}}(WV)\geq
\dim_{\mathrm{k}}(W),$ we have%
\begin{equation}
c(W)\geq(1-\lambda)\dim_{\mathrm{k}}(W). \label{maj_connecti}%
\end{equation}
We observe that when $\lambda<1$, $c(W)$ is always positive and takes a
discrete set of values.\ Therefore, when $\lambda<1$, there exists at least
one fragment and at least one atom. \emph{In the remainder of this paragraph
we will assume that }$\lambda<1$. Let $W_{1}$ and $W_{2}$ be two fragments
with nonzero intersection. By the previous lemma, we derive%
\[
c(W_{1}+W_{2})+c(W_{1}\cap W_{2})\leq c(W_{1})+c(W_{2})\leq2\kappa.
\]
Since $W_{1}+W_{2}$ and $W_{1}\cap W_{2}$ are finite-dimensional and not
reduced to $\{0\}$, we must have $c(W_{1}+W_{2})\geq\kappa$ and $c(W_{1}\cap
W_{2})\geq\kappa$.\ Hence $c(W_{1}+W_{2})=c(W_{1}\cap W_{2})=\kappa$. This
means that $W_{1}+W_{2}$ and $W_{1}\cap W_{2}$ are also fragments. If we
assume now that $W_{1}$ and $W_{2}$ are atoms, we obtain that $W_{1}=W_{2}$ or
$W_{1}\cap W_{2}=\{0\}$.

\begin{proposition}
\label{Prop_connect}Let $K$ be a division ring containing $\mathrm{k}$ in its
center and let $V$ be a finite-dimensional $\mathrm{k}$-subspace .

\begin{enumerate}
\item There exists a unique atom $H$ containing $1$ for $V$.

\item This atom is a division ring containing $\mathrm{k}$ in its center.

\item Moreover the atoms of $V$ are the right modules $xH$ where $x$ runs over
$K^{\ast}$.
\end{enumerate}
\end{proposition}

\begin{proof}
Since there exists at least one atom and the left translate of any atom is an
atom, there exists one atom $H$ containing $1$. Now, this atom must be unique
because atoms are equal or with an intersection reduced to $\{0\}$. In
particular, for any $x\in K$, $H=xH$ or $H\cap xH=\{0\}$. We claim that this
implies that $H$ is a division ring. Indeed, for any $h\in H$, $H\cap h^{-1}H$
contains $1$ for $1\in H$.\ Therefore $h^{-1}H=H$ and $H=hH$. So $H$ is stable
under multiplication. We then deduce that $H$ is a division ring as in the
proof of Lemma \ref{lem_stab}. Finally, given any atom $W$ of $V$, we must
have $w^{-1}W=H$ for any nonzero $w\in W$ since $H$ is the unique atom
containing $1$ and $w^{-1}H$ is an atom.
\end{proof}

\subsection{Tao's theorem for division rings}

\label{subsec-linTao}We say that $V=\mathrm{k}\langle A\rangle$, where $A$ is
a finite subset of $K^{\ast}$, is a \emph{space of small doubling}, when
$\dim_{\mathrm{k}}(A^{2})=O(\dim_{\mathrm{k}}(A)).\;$Simplest examples of
spaces of small doubling are the spaces $V=\mathrm{k}\langle A\rangle$
containing $1$ and such that $\dim_{\mathrm{k}}(A^{2})=\dim_{\mathrm{k}}(A)$.
Then by Lemma \ref{lem_stab}, $V$ is a division ring containing $\mathrm{k}$.
In general, a space of small doubling $\mathrm{k}\langle A\rangle$ is not a
division ring and neither a left nor right $H$-module for a division ring
$\mathrm{k}\subseteq H\subseteq K$. The following theorem, which is the linear
version of Theorem 1.2 in \cite{Tao}, permits to study the spaces of small
doubling in $K$.

\begin{theorem}
\label{Th_linTao}Let $K$ be a division ring containing the field $\mathrm{k}$
in its center. Consider finite-dimensional $\mathrm{k}$-subspaces $V$ and $W$
of $K$ such that $\dim_{\mathrm{k}}(W)\geq\dim_{\mathrm{k}}(V)$ and
$\dim_{\mathrm{k}}(WV)\leq(2-\varepsilon)\dim_{\mathrm{k}}(V)$ for some real
$\varepsilon$ such that $0<\varepsilon<2$. Then one of the following
statements holds:

\begin{itemize}
\item There exists a division ring $H$ containing $\mathrm{k}$ such that
$\dim_{\mathrm{k}}(H)\leq(\frac{2}{\varepsilon}-1)\dim_{\mathrm{k}}(V)$, and
$V$ is contained in a left module $Hx$ with $x\in K^{\ast}$.

\item There exists a division ring $H$ containing $\mathrm{k}$ such that
$\dim_{\mathrm{k}}(H)\leq(\frac{2}{\varepsilon}-1)/(\frac{2}{\varepsilon
}+1)\dim_{\mathrm{k}}(V)$, and a finite subset $X$ of $K^{\ast}$ with
$\left\vert X\right\vert \leq\frac{2}{\varepsilon}-1$ such that $V\subseteq
\bigoplus_{x\in X}$ $Hx$.
\end{itemize}
\end{theorem}

\begin{proof}
We apply linear Hamidoune connectivity with $\lambda=1-\frac{\varepsilon}{2}$.
We have by (\ref{maj_connecti}) $c(U)\geq\frac{\varepsilon}{2}\dim
_{\mathrm{k}}(U)$ for any $\mathrm{k}$-subspace $U$.\ This can be rewritten as%
\begin{equation}
\dim_{\mathrm{k}}(U)\leq\frac{2}{\varepsilon}c(U). \label{maj3}%
\end{equation}
We also get%
\[
c(W):=\dim_{\mathrm{k}}(WV)-(1-\frac{\varepsilon}{2})\dim_{\mathrm{k}}%
(W)\leq(2-\varepsilon)\dim_{\mathrm{k}}(V)-(1-\frac{\varepsilon}{2}%
)\dim_{\mathrm{k}}(V)=(1-\frac{\varepsilon}{2})\dim_{\mathrm{k}}(V).
\]
since $\dim_{\mathrm{k}}(WV)\leq(2-\varepsilon)\dim_{\mathrm{k}}(V)$ and
$\dim_{\mathrm{k}}(W)\geq\dim_{\mathrm{k}}(V)$. By Proposition
\ref{Prop_connect}, the unique atom containing $1$ is a division ring $H$.\ By
definition of an atom, we should have%
\[
\kappa=c(H)\leq c(W)\leq(1-\frac{\varepsilon}{2})\dim_{\mathrm{k}}(V).
\]
We therefore obtain, by using (\ref{maj3}) with $U=H$, that%
\[
\dim_{\mathrm{k}}(H)\leq\frac{2}{\varepsilon}c(H)\leq\frac{2}{\varepsilon
}c(W)\leq(\frac{2}{\varepsilon}-1)\dim_{\mathrm{k}}(V).
\]
If $V$ is contained in a left module $Hx$, we are done. Assume $V$ intersects
at least two such left $H$-modules. By using that $c(H)=\dim_{\mathrm{k}%
}(HV)-(1-\frac{\varepsilon}{2})\dim_{\mathrm{k}}(H)$ and the previous
inequality $c(H)\leq(1-\frac{\varepsilon}{2})\dim_{\mathrm{k}}(V)$, we get%
\begin{equation}
\dim_{\mathrm{k}}(HV)\leq(1-\frac{\varepsilon}{2})\dim_{\mathrm{k}%
}(V)+(1-\frac{\varepsilon}{2})\dim_{\mathrm{k}}(H). \label{ineq3}%
\end{equation}
Since $V$ intersects at least two left $H$-modules, we must have
$\dim_{\mathrm{k}}(HV)\geq2\dim_{\mathrm{k}}(H)$. By using (\ref{ineq3}), this
gives
\[
\dim_{\mathrm{k}}(H)\leq\frac{1-\frac{\varepsilon}{2}}{1+\frac{\varepsilon}%
{2}}\dim_{\mathrm{k}}(V)<\dim_{\mathrm{k}}(V).
\]
We can also bound $\dim_{\mathrm{k}}(V)$ by $\dim_{\mathrm{k}}(HV)$ in
(\ref{ineq3}). This yields
\begin{equation}
\dim_{\mathrm{k}}(HV)\leq(\frac{2}{\varepsilon}-1)\dim_{\mathrm{k}}(H)\text{.}
\label{ineq4}%
\end{equation}
Now $\mathrm{k}\langle HV\rangle$ is left $H$-invariant and finite-dimensional
because $\dim_{\mathrm{k}}(HV)\leq\dim_{\mathrm{k}}(H)\dim_{\mathrm{k}}%
(V)\leq\dim_{\mathrm{k}}(V)^{2}$. By Lemma \ref{lemma_stab}, there thus exists
a finite subset $X$ of $K^{\ast}$ such that
\[
\mathrm{k}\langle HV\rangle=\bigoplus_{x\in X}Hx\text{.}%
\]
By (\ref{ineq4}), we should have $\left\vert X\right\vert \leq(\frac
{2}{\varepsilon}-1)$. Moreover $V\subseteq\mathrm{k}\langle HV\rangle$, which
concludes the proof.
\end{proof}

\bigskip

\noindent\textbf{Remark:} When $\dim_{\mathrm{k}}(A^{2})\leq(2-\varepsilon
)\dim_{\mathrm{k}}(A)$, we can apply Theorem \ref{Th_linTao} with
$V=W=\mathrm{k}\langle A\rangle$ and obtain information on the covering of $V$
by left $H$-modules. This can be interpreted as a description of subspaces
with small doubling similar to the description of sets with small doubling
obtained in \cite{Tao}.

\bigskip

\noindent\textbf{Acknowledgments:} The author thanks both anonymous referees
for their very careful reading of the paper and their many valuable
corrections and comments.

\vfill
\noindent Laboratoire de Math\'ematiques et Physique Th\'eorique (UMR CNRS
6083)\newline Universit\'e Fran\c{c}ois-Rabelais, Tours \newline
F\'ed\'eration de Recherche Denis Poisson - CNRS\newline Parc de Grandmont,
37200 Tours, France. \newline\noindent{cedric.lecouvey@lmpt.univ-tours.fr}%
\newline

\end{document}